\documentclass[11pt, reqno]{amsart}

\usepackage{fullpage}
\usepackage{enumerate}
\usepackage{setspace}

\usepackage{amsmath,amsfonts,amssymb,graphics,amsthm, comment}
\usepackage{hyperref}
\hypersetup{
    colorlinks=true,
    linkcolor=blue,
    citecolor=red,
    urlcolor=blue,
    pdfborder={0 0 0}
}    
\usepackage{graphicx}
\usepackage{xcolor}

\usepackage{cleveref}

  \crefname{theorem}{Theorem}{Theorems}
  \crefname{thm}{Theorem}{Theorems}
  \crefname{lemma}{Lemma}{Lemmas}
  \crefname{lem}{Lemma}{Lemmas}
  \crefname{remark}{Remark}{Remarks}
  \crefname{prop}{Proposition}{Propositions}
  \crefname{proposition}{Proposition}{Propositions}
\crefname{notation}{Notation}{Notations}
\crefname{claim}{Claim}{Claims}
  \crefname{defn}{Definition}{Definitions}
  \crefname{corollary}{Corollary}{Corollaries}
  \crefname{section}{Section}{Sections}
  \crefname{figure}{Figure}{Figures}
  \crefname{exercise}{Exercise}{Exercises}
    \crefname{assumption}{Assumption}{Assumptions}

\newtheorem{thm}{Theorem}[section]

\newtheorem{lemma}[thm]{Lemma}
\newtheorem{corollary}[thm]{Corollary}
\newtheorem{prop}[thm]{Proposition}

\newtheorem{proposition}[thm]{Proposition}
\newtheorem{defn}[thm]{Definition}

\numberwithin{equation}{section}

\theoremstyle{definition}
\newtheorem{remark}[thm]{Remark}

\def \min {\text{min}}

\def\cM{\mathcal{M}}
\def\cL{\mathcal{L}}

\def\cH{\mathcal{H}}
\def\cG{\mathcal{G}}
\def\cF{\mathcal{F}}
\def\cE{\mathcal{E}}
\def\cD{\mathcal{D}}
\def\cC{\mathcal{C}}
\def\cB{\mathcal{B}}
\def\cA{\mathcal{A}}

\def \ve {\varepsilon}

\def \ms {\mathsf}

\def\P{\mathbb{P}}
\def\E{\mathbb{E}}

\def\R{\mathbb{R}}
\def\Z{\mathbb{Z}}
\def\N{\mathbb{N}}

\def\R{\mathbb{R}}

\def\1{\mathbf{1}}

\def  \p- {p\textunderscore}

\DeclareMathOperator{\dist}{dist}

\usepackage{extarrows}

\date{}

\title{Proper $3$-colorings of $\Z^2$ are Bernoulli}
\date{\today}

\author{Gourab Ray}
\address{Gourab Ray\hfill\break
	University of Victoria\\
	Department of Mathematics\\
	Victoria, BC, V8W 2Y2, Canada.}

\author{Yinon Spinka}
\address{Yinon Spinka\hfill\break
	University of British Columbia\\
	Department of Mathematics\\
	Vancouver, BC, V6T 1Z2, Canada.}

\AtEndDocument{
  \bigskip
  \small
  \textit{Emails:} \texttt{gourabray@uvic.ca, yinon@math.ubc.ca} \par
}

\begin{document}

\maketitle
\begin{abstract}
We consider the unique measure of maximal entropy for proper 3-colorings of $\Z^2$, or equivalently, the so-called zero-slope Gibbs measure.
Our main result is that this measure is Bernoulli, or equivalently, that it can be expressed as the image of a translation-equivariant function of independent and identically distributed random variables placed on $\Z^2$. Along the way, we obtain various estimates on the mixing properties of this measure.
\end{abstract}

{\section{Introduction}
A (proper) \textbf{$3$-coloring} of a graph is an assignment of one of three colors, say from $\{0,1,2\}$, to each of its vertices so that no two adjacent vertices receive the same color.
In this paper, we are concerned with 3-colorings of the square lattice $\Z^2$ with nearest-neighbor adjacency. Specifically, our goal is to show that a certain natural translation-invariant measure on the space of such 3-colorings is \textbf{Bernoulli}, meaning that it is isomorphic as a measure-preserving dynamical system to an i.i.d.\ process on $\Z^2$, i.e, there is an invertible measure-preserving map from an i.i.d.\ process to it, which is defined almost everywhere and commutes with all translations of $\Z^2$. Alternatively, being Bernoulli is equivalent to being a factor of an i.i.d.\ process.

A vertex of $\Z^2$ is even if the sum of its two coordinates is even.
Let $D$ be a finite subset of $\Z^2$ and let $\partial D$ denote its internal vertex boundary, namely, the set of vertices in $D$ that have a neighbor outside $D$.
Let $\mu^{01}_D$ be the uniform measure on the set of all 3-colorings of $D$ whose values on $\partial D$ are fixed to be 0 and 1 on even and odd vertices, respectively. We shall show that $\mu^{01}_D$ converges as $D \uparrow \Z^2$ (for sufficiently nice $D$ such as boxes) to a translation-invariant measure $\mu$ on 3-colorings of $\Z^2$.
It is this measure that we are concerned with here. Our main result is the following.


\begin{thm}\label{thm:main} 
The measure $\mu$ is Bernoulli.
\end{thm}
The key step in the proof is to convert the proper 3-colorings into height functions corresponding homomorphisms from $\Z^2$ to $\Z$ via a bijection.
Let us make several quick remarks. Firstly, the limiting measure $\mu$ does not depend on the specific choice of boundary condition used to define $\mu^{01}_D$. To be more precise, for $\xi \in \{0,1,2\}^{\partial D}$, let $\Omega^\xi_D$ be the set of all 3-colorings of $D$ which agree with $\xi$ on $\partial D$. We refer to $\xi$ as a boundary condition. Suppose that $\Omega^\xi_D$ is non-empty and let $\mu^\xi_D$ denote the uniform measure on $\Omega^\xi_D$. The convergence of $\mu^\xi_D$ to $\mu$ holds for a larger class of boundary conditions, namely, those whose oscillations (in terms of the associated height function) are of smaller order than the square-root of the logarithm of the in-radius of $D$; see \cref{rmk:sqrtlog}.


Standard arguments imply that $\mu$ is a Markov random field and a uniform Gibbs measure for 3-colorings of $\Z^2$, meaning that, for any finite set $D \subset \Z^2$ and any boundary condition $\xi \in \{0,1,2\}^{\partial D}$ such that $\mu(\Omega^\xi_D)>0$, conditioned on $\Omega^\xi_D$, the coloring on $D$ has distribution $\mu^\xi_D$ and is independent of the coloring on $\Z^2 \setminus D$.

Our last remark concerns the notion of a measure of maximal entropy.
We shall not define this notion precisely (and we shall not need it), but simply mention that it roughly means that the restriction of the measure to a large box has Shannon entropy which is nearly as large as possible on a volume scale. The topological entropy of 3-colorings of $\Z^2$ has been computed by Lieb~\cite{lleb2004residual} to be $\frac32 \log \frac43$, which means that the number of 3-colorings of an $n$-by-$n$ grid grows like $(4/3)^{\frac32 n^2 (1+o(1))}$ as $n \to \infty$.
It is known that $\mu$ is a measure of maximal entropy for 3-colorings of $\Z^2$~(see~\cite[Theorem~1.2]{galvin2012phase} for an elementary proof). Furthermore, it follows from the results in~\cite{sheffield2005random,chandgotia2018delocalization} (with some minor additional arguments) that there is a unique measure of maximal entropy for 3-colorings of $\Z^2$. Thus, our main theorem can be formulated concisely as \emph{the unique measure of maximal entropy for 3-colorings of $\Z^2$ is Bernoulli}.

The proof of \cref{thm:main} relies crucially on the Russo--Seymour--Welsh theory for homomorphisms from $\Z^2$ to~$\Z$ recently developed in~\cite{chandgotia2018delocalization,duminil2019logarithmic}.
This also allows us to establish a quantitative power-law mixing condition, which we regard as the main probabilistic content of this paper; see \cref{thm:colorings-are-locally-mixing}.

%
%

\medbreak

\noindent\textbf{Related results.}
Bernoullicity is a type of mixing condition. Several related mixing conditions for translation-invariant random fields on $\Z^d$ in increasing order of strength are ergodicity, weak mixing, (strong) mixing, $k$-fold mixing ($k \ge 3$), Bernoullicity, and finitary factor of iid.
For Markov random fields, an additional condition called K is equivalent to full tail triviality~\cite{den1997k} and is between $k$-fold mixing and Bernoullicity.
Slawny~\cite{slawny1981ergodic} gave examples of measures which are ergodic but not weakly mixing (for $d \ge 2$), measures which are weakly mixing but not mixing (for $d \ge 3$), and measures which are mixing but not 3-fold mixing (for $d \ge 3$).
In two dimensions, Ledrappier~\cite{ledrappier1978champ} gave a simple construction of a zero-entropy Markov random field which is mixing but not 3-fold mixing, and Hoffman~\cite{hoffman1999markov,hoffman2004family} constructed Markov random fields which are K but not Bernoulli.
Van den Berg and Steif~\cite{van1999existence} showed the existence of Markov random fields which are Bernoulli but not finitary factors of i.i.d.\ processes (for $d \ge 2$).

The question of determining whether a given random field is Bernoulli is non-trivial and has received much attention in the literature. We mention the following open question of van den Berg and Steif~\cite[Question~3]{van1999existence}: if a translation-invariant Markov random field is the unique Markov random field with its conditional probabilities, is it necessarily Bernoulli? It is remarked there that such a Markov random field is known to be K (equivalently, full tail trivial).

Let us discuss the situation for 3-colorings of $\Z^d$ for $d \ge 3$. In general, the convergence of the finite-volume measures $\mu^{01}_D$ (defined as before) is not known. In sufficiently high dimensions, convergence holds along `nice' domains of a given parity (so-called even or odd domains)~\cite{feldheim2015long}. The limiting measures in this case depend on the parity of the domain~\cite{peled2017high,galvin2012phase}. In particular, these measures are not translation-invariant, but rather invariant only with respect to parity-preserving translations (or other automorphisms). This leads to the fact that the translation-invariant measures of maximal entropy are not mixing and thus also not Bernoulli (as they are mixtures of distinct extremal measures). Nevertheless, when restricting to the action of the subgroup of parity-preserving translations, these measures become Bernoulli (in fact, they are quite weak Bernoulli with exponential rate; see~\cite[Lemma~6.4]{feldheim2015long}).

Let us also remark about the situation for $q$-colorings of $\Z^d$ for $q \ge 4$ and $d \ge 2$. For a fixed number of colors $q$, in high enough dimensions ($d \ge Cq^{10}\log^3 q$ suffices), the situation is similar to that of $3$-colorings in that the measures obtained from fixed-color boundary conditions are not translation-invariant, but they are invariant to parity-preserving automorphisms~\cite{peled2018rigidity}. On the other hand, in any given dimension, when the number of colors is large ($q>4d$ suffices), there is a unique Gibbs measure, the finite-volume measures (with any boundary conditions) converge to this measure, and this measure is translation-invariant and strong spatial mixing (this all follows from Dobrushin's uniqueness condition; see, e.g., \cite{peled2020lectures}).
Consequently, this measure is Bernoulli, and in fact, also a finitary factor of an i.i.d.\ process~\cite{spinka2018finitarymrf}.
In two dimensions, it is known that $q$-colorings satisfy strong spatial mixing for any $q \ge 6$~\cite{achlioptas2005rapid,goldberg2006improved}, and hence, in this case, the unique Gibbs measure is also a finitary factor of an i.i.d.\ process~\cite{spinka2018finitarymrf}. The situation is still open for 4 and 5 colors in two dimensions. It would be interesting to determine whether the measure on 3-colorings of $\Z^2$ studied here is also a finitary factor of an i.i.d.\ process. This is a particular instance of a question raised by Steif~\cite[Question~18.2]{boyle2008open}: if a subshift of finite type (in $\Z^d$, $d \ge 2$) has a unique measure of maximal entropy which is Bernoulli, must it be a finitary factor of an i.i.d.\ process?

Proper $q$-colorings may be viewed as the zero-temperature antiferromagnetic $q$-state Potts model. The 2-state Potts model is known as the Ising model.
Ornstein and Weiss~\cite{OW73} (see also~\cite{adams1992folner}) showed that the plus state of the ferromagnetic Ising model on~$\Z^d$ at any positive temperature is Bernoulli.
H{\"a}ggstr{\"o}m, Jonasson and Lyons~\cite{haggstrom2002coupling} extended this to the ferromagnetic Potts model for any $q \ge 2$.
Our result shows this for (a certain Gibbs measure of) the zero-temperature antiferromagnetic 3-state Potts model. It is natural to expect that this extends to positive temperature. However, as our proof relies crucially on the height function representation for 3-colorings, which does not extend to positive temperature, we are unable to answer this question.


\section{Locally mixing measures}
\label{sec:locally-mixing-measures}

Several conditions related to Bernoullicity have been introduced in the literature, among which are weak Bernoulli, very weak Bernoulli, quite weak Bernoulli and F{\o}lner independence. In this section, we introduce a new notion, which we call local mixing, and present some general results about it. The results in this section apply to random fields on $\Z^d$ in any dimension~$d$. All measures here are probability measures.

It will be useful for us to define the notion of local mixing for a family of measures, rather than just for a single measure.
A \textbf{rate function} is any decreasing function $\rho \colon \N \to [0,\infty)$ such that $\rho(k) \to 0$ as $k \to \infty$. Let $\Lambda_n$ denote the box $[-n,n]^d \cap \Z^d$.

\begin{defn}(locally mixing)\label{def:local-mixing}
Let $\cA$ be finite and let $\cM$ be a collection of pairs $(\mu,D)$ such that $D \subset \Z^d$ and $\mu$ is a measure on $\cA^D$. We say that $\cM$ is \textbf{locally mixing} with rate function $\rho$ if the following holds. Let $n \ge 1$ and let $(\mu,D),(\mu',D') \in \cM$ be such that $\Lambda_n \subset D \cap D'$. Then there exists a coupling of $f \sim \mu$ and $f' \sim \mu'$ such that
\begin{enumerate}
 \item $f_{|D \setminus \Lambda_n}$ and $f'$ are independent.
 \item $\P(f(v) \neq f'(v)) \le \rho(n-k)$ for any $0 \le k \le n$ and $v \in \Lambda_k$.
\end{enumerate}
\end{defn}

We say that $\cM$ is locally mixing if it locally mixing with some rate function.
We stress that there is no restriction on the sets $D$, and, in particular, some of them  may be infinite. Clearly, if a family $\cM$ is locally mixing, then so is any subset of it (with the same rate function).

The notion of local mixing also makes sense for a single measure.
We say that a measure $\mu$ on $\cA^{\Z^d}$ is locally mixing if $\cM=\{(\mu,\Z^d)\}$ is. The following simple proposition shows that a locally mixing family gives rise to a unique limiting measure, which is itself locally mixing.

\begin{prop}\label{prop:locally-mixing-has-limit}
Let $\cM=\{(\mu_i,D_i)\}_{i=1}^\infty$ be locally mixing with $D_i \uparrow \Z^d$ as $i \to \infty$. Then $\mu_i$ converges as $i \to \infty$ to a measure $\mu$ on $\cA^{\Z^d}$ which is locally mixing with the same rate function.
\end{prop}
\begin{proof}
Suppose that $\cM$ is locally mixing with rate function $\rho$. To establish the convergence, it suffices to show that, for any $i,j,k,n$ such that $\Lambda_k \subset \Lambda_n \subset D_i \cap D_j$,
\[ \dist_{\text{TV}}((\mu_i)_{|\Lambda_k}, (\mu_j)_{|\Lambda_k}) \le |\Lambda_k| \cdot \rho(n-k) .\]
Indeed, sampling $f \sim \mu_i$ and $f' \sim \mu_j$ from a coupling as in the definition of local mixing, the inequality follows by a union bound. Let $\mu$ be the limiting measure. It is straightforward from the convergence that $\cM \cup \{(\mu,\Z^d)\}$ is locally mixing. In particular, $\mu$ is locally mixing.
\end{proof}

A key property of local mixing is that it implies two other mixing properties (for translation-invariant measures): full tail triviality (which implies strong mixing and ergodicity) and Bernoullicity. The former (which we do not define here) easily follows from the definition of local mixing. The latter is stated in the following proposition.

\begin{prop}\label{prop:locally-mixing-is-Bernoulli}
Any translation-invariant locally mixing measure on $\cA^{\Z^d}$ is Bernoulli.
\end{prop}


Our strategy for proving \cref{prop:locally-mixing-is-Bernoulli} is to verify a classic condition called \textbf{very weak Bernoulli}, which is known to be equivalent to Bernoulli. In fact, we verify a stronger condition called \textbf{F{\o}lner independence}~\cite{adams1992folner}, which we now proceed to define.


We begin by defining the so-called $\bar d$-distance between two measures $\nu$ and $\lambda$ on $\cA^V$, with $\cA$ and $V$ finite. Roughly speaking, the $\bar d$-distance is small if one can couple samples of $\nu$ and $\lambda$ so that they tend to agree on most elements of $V$. Precisely, the \textbf{$\bar d$-distance} between $\lambda$ and $\nu$ is
$$
\bar d(\nu, \lambda) := \inf_{\substack{(X,Y)\\X\sim \nu, Y \sim \lambda}}  \left \{ \frac1{|V|} \sum_{v \in V} \P({X(v) \neq Y(v)}) \right \},
$$
where the infimum is taken over couplings of random variables $X$ and $Y$ with distributions $\nu$ and $\lambda$, respectively.

Consider now a translation-invariant measure $\mu$ on $\cA^{\Z^d}$. Loosely speaking, $\mu$ is F{\o}lner independent if for most conditionings outside a large box, the conditional measure inside the box is close in $\bar d$-distance to the unconditional measure inside the box. Given a set $U \subset \Z^d$, we denote by $\mu|_U$ the restriction of $\mu$ to $U$. Given a finite set $B \subset \Z^d$ and a feasible $\xi \in \cA^B$ (by feasible we mean that $\mu(\xi)>0$), we denote by $\mu|^\xi_U$ the restriction of $\mu$ to $U$ when conditioned on $\xi$. Here and throughout the paper, we sometimes identify a configuration $\xi \in \cA^B$ with the event $\{ f \in \cA^{\Z^d} : f_{|B}=\xi\}$.

\begin{defn}[F{\o}lner independence]\label{def:FI}
A translation-invariant measure $\mu$ on $\cA^{\Z^d}$ is \textbf{F{\o}lner independent} if for all $\ve>0$ there exists $N$ such that for any $n \ge N$ and finite $S \subset \Z^d \setminus \Lambda_n$,
\begin{equation}\label{eq:vwb}
\bar d (\mu|^\xi_{\Lambda_n}, \mu|_{{\Lambda_n}}) <\ve
\end{equation}
for all feasible $\xi \in \cA^S$, except for a set of $\mu$-measure at most $\ve$.
\end{defn}

We mention that the very weak Bernoulli condition is defined similarly, with the only difference being that $S$ is required to be a subset of a certain ``lexicographical past'' of $\Lambda_n$. We do not give the precise definition here, but content ourselves with the fact that, by its definition, it is weaker than F{\o}lner independence (for Markov random fields, the two are in fact equivalent~\cite{den1997k}). We rely on the following important theorem, which is due to Ornstein~\cite{ornstein1974ergodic} and Ornstein--Weiss~\cite{ornstein1974finitely} in the one-dimensional case. For the general case, we refer to~\cite{kammeyer90,katznelson1972commuting}.



\begin{thm}[\cite{kammeyer90,katznelson1972commuting}]
A translation-invariant ergodic measure on $\cA^{\Z^d}$ is Bernoulli if and only if it is very weak Bernoulli. In particular, if it is F{\o}lner independent, then it is Bernoulli.
\end{thm}

Thus, \cref{prop:locally-mixing-is-Bernoulli} will follow once we establish the simple fact that local mixing implies F{\o}lner independence. In fact, as it turns out, the two are actually equivalent. As this latter statement requires more work to establish and as it is not our main concern, we postpone its proof to the end of the paper (\cref{sec:FI}).

\begin{prop}\label{prop:locally-mixing-and-FI}
A translation-invariant measure on $\cA^{\Z^d}$ is locally mixing  if and only if it is F{\o}lner independent.
\end{prop}

\begin{proof}[Proof of first half of \cref{prop:locally-mixing-and-FI} (local mixing implies F{\o}lner independence)]
Let $\mu$ be a translation-invariant locally mixing measure on $\cA^{\Z^d}$.
We show that for any $n \ge m \ge 1$ and finite $S \subset \Z^d \setminus \Lambda_n$, the set $\cG \subset \cA^S$ of feasible configurations $\xi$ such that
\begin{equation}\label{eq:vwb2}
\bar d (\mu|^\xi_{\Lambda_n}, \mu|_{{\Lambda_n}}) \le \ve := \sqrt{\rho(n-m) + \tfrac{|\Lambda_n \setminus \Lambda_m|}{|\Lambda_n|}}
\end{equation}
satisfies $\mu(\cG) \ge 1-\ve$.
Since $\ve$ can be made arbitrarily small by taking $n$ large enough and $m=n-\sqrt{n}$, say, this will establish that $\mu$ is F{\o}lner independent.
Let $f,f' \sim \mu$ be sampled from a coupling as in the definition of local mixing.
Let $X$ denote the average number of vertices $v \in \Lambda_n$ such that $f(v) \neq f'(v)$. Then
\[ \E X = \frac{1}{|\Lambda_n|} \sum_{v \in \Lambda_n} \P(f(v) \neq f'(v)) \le \rho(n-m) + \tfrac{|\Lambda_n \setminus \Lambda_m|}{|\Lambda_n|} = \ve^2 .\]
Define $Y := \E[X \mid f_{|S}]$ and note that $\E Y = \E X \le \ve^2$.
Thus, Markov's inequality yields that $\P(Y \ge \ve) \le \ve$. Finally, since $f_{|\Z^d \setminus \Lambda_n}$ and $f'$ are independent, the conditional distribution of $f'_{|\Lambda_n}$ given $f_{|S}$ is $\mu|_{\Lambda_n}$. It follows that the event $\{Y<\ve\}$ is contained in the event $\{f_{|S} \in \cG\}$, and hence, that $\mu(\cG) \ge 1-\ve$.
\end{proof}

\section{Proof outline}
\label{sec:pre}

Given the results about locally mixing measures discussed in \cref{sec:locally-mixing-measures}, our main result on the Bernoullicity of the 3-coloring measure will follow by showing that the appropriate family of 3-coloring measures is locally mixing.

We extend the definition of $\mu^{01}_D$ to allow for any two fixed boundary colors. Specifically, for distinct $i,j \in \{0,1,2\}$, let $\mu^{ij}_D$ be the uniform measure on the set of all 3-colorings of $D$ whose values on $\partial D$ are fixed to be $i$  and $j$ on even and odd vertices, respectively. A subset $D$ of $\Z^2$ is \textbf{simply connected} if it is connected and its complement $\Z^2 \setminus D$ is connected. A rate function $\rho$ is a \textbf{power-law rate function} if it satisfies $\rho(n) \le Cn^{-\alpha}$ for some $C,\alpha>0$ and all $n \ge 1$.

\begin{thm}\label{thm:colorings-are-locally-mixing}
Let $\cM$ be the family of all pairs $(\mu^{ij}_D,D)$ with $i,j \in \{0,1,2\}$ distinct and $D \subset \Z^2$ finite and simply connected. Then $\cM$ is locally mixing with a power-law rate.
\end{thm}

We remark that a power-law rate is best possible; see \cref{cor:correlations}.
The proof of \cref{thm:colorings-are-locally-mixing} also shows that the larger family consisting of all pairs $(\mu^\xi_D,D)$ with $D$ finite and simply connected and with $\xi$ having bounded oscillations (in terms of the associated height function) is also locally mixing with a power-law rate; see \cref{rmk:sqrtlog}.

Together with \cref{prop:locally-mixing-has-limit}, the theorem immediately implies that $\mu^{ij}_D$ converges to a measure $\mu$ on 3-colorings of $\Z^2$ as $D$ increases to $\Z^2$ along simply connected finite subsets. Furthermore, since the limit does not depend on $i$ and $j$ (as $\cM$ contains pairs with every choice of $i$ and~$j$), it follows that $\mu$ is invariant to permutations of the colors.
It is also straightforward that $\mu$ is invariant to any automorphism $T$ of $\Z^2$, since $\mu^{01}_D \circ T^{-1}$ is either $\mu^{01}_{T(D)}$ or $\mu^{10}_{T(D)}$, both of which belong to $\cM$.

\begin{proof}[Proof of~\cref{thm:main}]
\cref{thm:colorings-are-locally-mixing} and \cref{prop:locally-mixing-has-limit} imply that $\mu$ is locally mixing, and \cref{prop:locally-mixing-is-Bernoulli} implies that $\mu$ is Bernoulli.
\end{proof}

The rest of the paper is mostly focused on proving \cref{thm:colorings-are-locally-mixing}. Below we introduce the height function representation for 3-colorings and give preliminary results for it in \cref{sec:preliminaries}. We then proceed to describe the proof strategy in \cref{sec:proof-strategy}. The full proof is given in \cref{sec:proof}.

\subsection{The height function}

It is well known that $3$-colorings of $\Z^2$ can be represented as homomorphisms from $\Z^2$ to $\Z$.
Recall that a homomorphism $\varphi$ from $D \subseteq \Z^2$ to $\Z$ is a map such that $|\varphi(x)-\varphi(y)|=1$ whenever $x,y \in D$ are adjacent. We also refer to such homomorphisms as height functions. We always work with height functions which are even on the even sublattice of $\Z^2$.

Suppose that $D$ is simply connected.
It can be checked that the mapping $h \mapsto h \mod 3 $ is a bijection from the space of height functions whose value is fixed on some vertex of $D$ to the space of 3-colorings of $D$ whose color is fixed on that same vertex. The inverse mapping can be defined as follows. If $\ms c$ is a $3$-coloring, then the gradient of $h$ along the directed edge $e=(u,v)$ is given by $\nabla h (e)  =+1$ if $\ms c(v) - \ms c(u) \in \{1,-2\}$ and $-1$ if  $\ms c(v) - \ms c(u) \in \{-1,2\}$. In particular, given a $3$-coloring of $\Z^2$, this mapping defines the gradient of a height function $h$ on $\Z^2$. In general, height functions will come with a predefined set of boundary conditions (precise definition will follow) which will always determine the function from its gradient.

Suppose that $D$ is finite and simply connected. Let $\phi_D^{01}$ denote the uniform measure on height functions on $D$ whose values on the boundary $\partial D$ are fixed to be $0$ and 1 on even and odd vertices, respectively. Then $\phi_D^{01}$ is pushed forward to $\mu^{01}_D$ by the modulo 3 map.

\subsection{Preliminary results}
\label{sec:preliminaries}

We begin with some notation.
For any $m$ and $n$, denote $\Lambda_{m,n} = [-m,m] \times [-n,n] \cap \Z^2$, and recall that $\Lambda_n=\Lambda_{n,n}$. For $n \ge m$, let $A_{m,n}$ denote the annulus $\Lambda_n \setminus \Lambda_m$.
A \textbf{path} in $\Z^2$ is a sequence of vertices $(v_1,v_2,\ldots, v_n)$ such that $v_i$ is adjacent to $v_{i+1}$ for $i=1,\ldots, n-1$. A path is a \textbf{loop} if $v_i \neq v_j$ for all $1\le i \neq j \le n-1$ and $v_1 = v_n$. We also require a notion of diagonal connectivity in $\Z^2$: say that $u$ is a $\times$-neighbor of $v$ if $u$ and $v$ are at Euclidean distance $\sqrt{2}$ (i.e., they are diagonal neighbors). We define a $\times$-path ($\times$-loop) in a similar way, replacing the $\Z^2$ adjacency by $\times$-adjacency. Note that a $\times$-path consists of vertices of a single parity, so that we may talk about even and odd $\times$-paths. A \textbf{domain} is a simply connected finite subset $D$ of $\Z^2$ whose boundary $\partial D$ is entirely contained in the \emph{even} lattice. When $D$ is a domain, we write $\mu^0_D$ and $\phi^0_D$ for $\mu^{01}_D$ and $\phi^{01}_D$, respectively.

A key tool we need is the recently established Russo--Seymour--Welsh estimate for height functions~\cite{duminil2019logarithmic} (a classical estimate in planar percolation-type models), which was used in~\cite{duminil2019logarithmic} to show logarithmic variance for the height function.
Let $\cH_{h\ge k}(\Lambda_{\rho n,n})$ (resp.\ $\cH_{h=k}^\times(\Lambda_{\rho n,n})$) be the event that there is a path (resp.\ $\times$-path) with height at least $k$ (resp.\ equal to $k$) joining the left and right boundaries of $\Lambda_{\rho n,n}$ and lying completely inside $\Lambda_{\rho n,n}$.

\begin{thm}[Russo--Seymour--Welsh estimate \cite{duminil2019logarithmic}] 
For every $\varepsilon,R,\rho,k>0$, there exists $c=c(\varepsilon,R,\rho,k)>0$ such that for any $n \ge \frac{10k}{\ve \wedge \rho}$ and any domain $\Lambda_{\rho n,n} \subset D\subset \Lambda_{Rn}$ such that the distance between $\Lambda_{\rho n,n}$ and $\partial D$ is at least $\varepsilon n$, 
 \begin{align}
c~\le~ &\phi_D^0 [\cH_{h\ge k}(\Lambda_{\rho n,n})]~\le~1- c,\label{eq:ha1}\\
c~\le~&\phi_D^0 [\cH_{h=k}^\times(\Lambda_{\rho n,n})]~\le~1-c.\label{eq:ha2}
\end{align}
\end{thm}

There are two key tools used in proving the above theorem, versions of which were also classically used to study planar percolation and random-cluster models \cite{duminil2017lectures}. The first one is duality of paths, which roughly states that a left-to-right crossing of height $h \ge m$ is blocked by a top-to-bottom crossing of height $h<m$. However, one needs to be careful with the type of connectivity used, and we will carefully point this out at the relevant place in the proof (rather than stating a general duality lemma, e.g., \cite[Lemma 2.4]{duminil2019logarithmic}, which we do not need).

The second key tool is a monotonicity property for the height function, classically known as the Fortuin--Kasteleyn--Ginibre \textbf{(FKG)} inequality and lattice condition. To properly state this, we introduce a general notion of boundary condition.
Given a domain $D$, a \textbf{boundary condition} is a pair $(B,\kappa)$ with $B \subset D$ and $\kappa$ a function that assigns a subset $\kappa_v \subset \Z$ to each $v \in B$.
Let ${\rm Hom}( D,B,\kappa)$ denote the set of homomorphisms $h$ on $D$ such that $h_v\in \kappa_v$ for every $v\in B$. We say that the boundary condition $(B,\kappa)$ is \textbf{admissible} if ${\rm Hom}( D,B,\kappa)$ is nonempty and finite. For an admissible boundary condition $(B,\kappa)$, we let $\phi^{B, \kappa}_D$ denote the uniform measure on ${\rm Hom}(D,B,\kappa)$. When $B=\partial D$, we omit $B$ from the notation.
A function $F\colon \Z^D \mapsto \R$ is \textbf{increasing} if $F(h) \ge F(h')$ for any $h, h' \in   \Z^D$  satisfying $h_v \ge h'_v$ for all $v \in  D$.  

 \begin{prop}[monotonicity for $h$~\cite{duminil2019logarithmic}]\label{prop:FKG_h}
 Consider a domain $D$ and two admissible boundary conditions $(B,\kappa)$ and $(B,\kappa')$ satisfying that for every $v\in B$, $\kappa_v=[a_v,b_v]$ and $\kappa'_v=[a'_v,b'_v]$ with $a_v\le a'_v$ and $b_v\le b'_v$ (the previous integers may be equal to $\pm\infty$). Then
 \begin{itemize}
 \item For every increasing function $F$, 
 $\phi^{B, \kappa'}_D[F(h)]\ge\phi^{B, \kappa}_D[F(h)]$;
 \item For any two increasing functions $F,G$, $\phi^ {B, \kappa}_D[F(h)G(h)] \ge   \phi^ {B, \kappa}_D[F(h)]  \phi^ {B, \kappa}_D [G(h)]$.
 \end{itemize}
 \end{prop}
 The first property is called the {\em comparison between boundary conditions} and the second the {\em FKG inequality}. We also crucially use monotonicity properties of $|h|$ in addition to those of $h$. We say that the boundary condition $(B,\kappa)$ is \textbf{$|h|$-adapted} if there exists a partition $B_{\rm pos}(\kappa) \sqcup B_{\rm sym}(\kappa)$ of $B$ such that
\begin{itemize}
\item for any $v \in B_{\rm pos}(\kappa)$, $\kappa_{v}\subset\Z_+:=\{0,1,2,\dots\}$;
\item for any $w \in B_{\rm sym}(\kappa)$, $\kappa_w =-\kappa_w$. 
\end{itemize}
We adopt the convention that $v \in B_{\rm sym}(\kappa)$ whenever $\kappa_v = \{0\}$, so there is no ambiguity in the above partition.

   \begin{prop}[monotonicity for $|h|$~\cite{duminil2019logarithmic}]\label{prop:FKG_modh}
    Consider a domain $D$ and two admissible $|h|$-adapted boundary conditions $(B,\kappa)$ and $(B,\kappa')$ satisfying $B_{\rm pos}(\kappa) \subseteq B_{\rm pos}(\kappa')$ and for every $v\in B$, $\kappa_v\cap\Z_+=[a_v,b_v]$ and $\kappa'_v\cap\Z_+=[a'_v,b'_v]$ with $a_v\le a'_v$ and $b_v\le b'_v$. Then
\begin{itemize}
 \item For every increasing function $F$, 
 $\phi^{B, \kappa'}_D[F(|h|)]\ge\phi^{B, \kappa}_D[F(|h|)]$;
 \item For any two increasing functions $F,G$, $\phi^ {B, \kappa}_D[F(|h|)G(|h|)] \ge   \phi^ {B, \kappa}_D[F(|h|)]  \phi^ {B, \kappa}_D[G(|h|)]$.
 \end{itemize}
 \end{prop}

It is these monotonicity properties that make working with the height function representation beneficial for us. 

\subsection{Proof strategy}\label{sec:proof-strategy}

Given a height function $h$, a \textbf{level loop} is a $\times$-loop on which the height is constant.
We claim that the Russo--Seymour--Welsh estimate guarantees that under $\phi^0_D$, for a domain $D$ containing $\Lambda_n$ but not containing a much larger box, most points in $\Lambda_n$ are surrounded by a level loop of height 0 contained entirely in $\Lambda_n$. Indeed, if we look at exponential scales starting from the  boundary inwards, there is a good chance of finding many successive level loops with height increment $\pm 2$. Using the symmetry of the height function, we can argue that the actual increment is $+2$ or $-2$ with equal chance, and hence the heights along these nested loops constitute a simple random walk. Since simple random walk is recurrent, there is a high chance of hitting a loop of height zero, thereby yielding the claim.

Now suppose we have two measures on height functions on $D$, one with boundary condition $\xi$ and one with 0 boundary condition (here we think of $\xi$ as specifying only the absolute values on the boundary). By FKG for $|h|$, we can couple samples $h^\xi$ and $h^0$ from these measures so that inside the outermost level loops of $0$ for $h^\xi$, the two height functions agree. Indeed, FKG for $|h|$ implies that $|h^\xi|$ stochastically dominates $|h^0|$, which means that we can couple $h^\xi$ and $h^0$ so that $|h^\xi|$ dominates $|h^0|$ pointwise. However, since 0 is the lowest absolute value, the outermost level loop with height 0 for $h^\xi$ must also be a height 0 level loop for $h^0$, and hence, by the domain Markov property, we can couple them inside to be the same. There are some technical issues with applying this argument as is, but these can be handled using standard exploration procedures.

However, combining the above two arguments is not enough, since when the boundary condition $\xi$ is very high in absolute value, there is a good chance that there is no outermost level loop of height 0, and indeed the height function delocalizes as $\Lambda_n$ becomes large~\cite{chandgotia2018delocalization,duminil2019logarithmic}. Thus, we need to reduce to the case where the boundary condition $\xi$ can be taken to be not too large. For this, we show that with high probability in the Gibbs measure of the coloring, we can find a color 0 loop $\cL$ in the annulus $A_{n,\Delta n}$ for some large $\Delta$. Now, by the domain Markov property, the law of the coloring inside the domain enclosed by $\cL$ can be viewed as a uniform homomorphism height function with 0 boundary conditions on $\cL$. Therefore, with high probability, we have somewhat reduced to the desired case above (the boundary condition is on the loop rather than on the boundary of the box).
Note that we need to switch back to the height function representation of the coloring to establish the coupling since we are crucially using FKG for absolute value of the height function to that end.

By employing a two-step iteration of the above argument, one may establish the required coupling needed to verify the local mixing condition. However, this would not yield a power-law rate function. The reason for this stems from the fact that, with probability that is polynomial in $k$, a simple random walk does not return 0 within $k$ steps, whereas the number of steps $k$ is with high probability only logarithmic in $n$, so that the former argument would yield a rate function $\rho$ decaying like $\rho(n) \le C(\log n)^{-\alpha}$. To obtain a power-law rate, we apply the above argument iteratively, switching back and forth between the coloring and height function representations. Such a back-and-forth procedure seems essential for this (see \cref{cor:correlations}).
Roughly speaking, we first find the outermost monochromatic loop in the coloring. We then we switch to the height function and start looking for a level loop of height 0, allowing ourselves to stop if we do not find such a loop after searching a constant number of scales. If we find such a loop quickly, then we are done. Otherwise, we give up on finding the height 0 level loop in this iteration, and proceed to the next iteration by going back to the coloring in order to find the next outermost monochromatic loop. Repeating this procedure yields a coupling which establishes local mixing with a power-law rate.

\section{Proof of Theorem~\ref{thm:colorings-are-locally-mixing}}
\label{sec:proof}

We begin with some results about the level lines of the height function in \cref{sec:level-lines}, and then proceed to give the proof of \cref{thm:colorings-are-locally-mixing} in \cref{sec:proof-main-thm}.

\smallskip
\noindent{\bf Policy on constants:} In the rest of the paper, we employ the following policy on constants. We write $C,c,C',c'$ for positive absolute constants, whose values may change from line to line. Specifically, the values of $C,C'$ may increase and the values of $c,c'$ may decrease from line to line.

\subsection{Level lines of the height function}
\label{sec:level-lines}

In this section, we prove some results about the level loops of the height function. Recall that a level loop of a height function is a $\times$-loop on which the height is constant. We will mostly be focused on level loops in the even lattice, or equivalently, level loops on which the height is even.

Let $D$ be a domain containing the origin and let $h$ be a height function on $D$ sampled from $\phi^0_D$. We inductively define a sequence of nested level loops $\cL_0,\dots,\cL_M$ surrounding the origin as follows.
First, we set $\cL_0$ to be the boundary loop of the domain $D$.
Next, we let $\cL_1$ be the outermost level loop with $|h| = 2$ surrounding the origin.
Now suppose that we have already defined the level loop $\cL_m$ and that the height on it is $H_m\in 2\Z$. Let $\cL_{m+1}$ be the outermost level loop surrounding the origin in the domain enclosed by $\cL_m$ and having height $H_{m+1}$ satisfying $|H_{m+1}-H_m| = 2$. If no such loop exists, then we set $M=m$ and stop the inductive procedure.
We say that the level loops $\cL_0,\dots,\cL_M$ are \textbf{essential}, and note that between $\cL_m$ and $\cL_{m+1}$ there may be many non-essential level loops of height $H_m$ surrounding the origin.

Observe that, conditioned on the exploration up to $\cL_m$, the height difference $H_{m+1}-H_m$ is uniform in $\{-2,2\}$. Indeed, given the exploration, the map $h \mapsto 2H_m - h$ inside the domain enclosed by $\cL_m$ is an involution which inverts the sign of $H_{m+1}-H_m$. Therefore, conditioned on the entire sequence of essential level loops $\cL_0,\dots,\cL_M$, the random variables $\{H_{m+1}-H_m\}_{m=0}^{M-1}$ are i.i.d.\ uniform $\pm 2$.
In particular, conditioned on $M$, the sequence $(\frac 12 H_m)_{m=0}^M$ defines an $M$-step simple symmetric random walk started from 0.

Our main technical result about the level lines of the height function is the following, which shows that, with high probability, the number of essential level loops in an annulus is linear in the logarithm of the aspect ratio of the outer and inner boundaries of the annulus. More specifically, we provide an upper bound on the number of such loops that intersect the annulus and a lower bound on the number of such loops that are entirely contained in the annulus.

\begin{proposition}\label{lem:level-lines-main}
There exist constants $C,c>0$ such that the following holds.
Let $k \ge 200$, let $a \ge 1$ and let $D$ be a domain.
Let $N$ be the number of essential level loops contained in $A_{k,2^ak}$ and let $N' \ge N$ be the number of essential level loops intersecting $A_{k,2^ak}$. 
Then
\begin{equation}
\phi_D^0(N' \ge n) \le e^{-cn \log (\frac na \wedge k)} \qquad\text{for any }n \ge Ca \label{eq:small}
\end{equation}
and, if $D$ contains $\Lambda_{2^ak}$, then
\begin{equation}\label{eq:large}
\phi_D^0(N \le ca) \le Ce^{-ca} .
\end{equation}
\end{proposition}

We point out that there is no assumption on the domain in~\eqref{eq:small}.

\subsubsection{Corollaries to Proposition~\ref{lem:level-lines-main}}
We give four corollaries to \cref{lem:level-lines-main} here. The first two will be used for the proof of \cref{thm:colorings-are-locally-mixing}, whereas the second two will not and simply provide additional information.
The former two focus on level loops whose heights are multiples of 6. Such level loops correspond in the coloring representation to monochromatic loops of color 0 in the even lattice. This conforms with the convention that the height function is even on the even lattice and with the fact that the coloring is the modulo 3 of the height function.

The first corollary shows that the origin is typically surrounded by a level loop of height 0 modulo 6 which lies inside an annulus of a fixed aspect ratio. This will later allow us to relate various boundary conditions which are arbitrarily far away from a given box to zero boundary conditions which are near the boundary of the box.

\begin{corollary}\label{lem:monochromatic_loop}
There exist $C,\alpha>0$ such that the following holds. Let $k \ge 1$, let $\Delta>1$ and let $D$ be a domain containing $\Lambda_{\Delta k}$. Then
\[ \phi^0_D(\Lambda_k\text{ is surrounded by a level loop of height 0 modulo 6 inside }\Lambda_{\Delta k}) \ge 1-C\Delta^{-\alpha}. \]
Moreover, the same bound holds under the measure $\phi^\kappa_D$ for any finite simply connected set $D$ containing $\Lambda_{\Delta k}$ and any admissible boundary condition $\kappa$ on $\partial D$ such that $\bigcup_{v \in \partial D} \kappa_v$ is contained in some interval of size $10$.
\end{corollary}
\begin{proof}
Recall that, given $M$, the sequence $(\frac 12 H_m)_{m=0}^M$ defines an $M$-step simple symmetric random walk.
Let $N$ be the number of essential level loops contained in $A_{k,\Delta k}$.
Since an $n$-step simple symmetric random walk (started from anywhere) visits $3\Z$ with probability at least $1-Ce^{-cn}$, it suffices to show that $\phi^0_D(N \le c\log\Delta) \le \Delta^{-\alpha}$.
This in turn follows from~\eqref{eq:small}.

Consider now the general case and suppose that $\kappa_v \subset [a+1,b-1]$ for all $v$ and some $a,b \in 2\Z$ having $b-a=20$.
Let $\cD$ be the largest domain (with respect to inclusion) contained in $D$ and containing $\Lambda_k$. It follows that every $v \in \partial \cD$ is at distance at most one from $\partial D$. In particular, under $\phi^\kappa_D$, the height on $\partial \cD$ must be between $a$ and $b$ everywhere.
Thus, by FKG, we can couple $h^- \sim \phi^a_\cD$, $h^+ \sim \phi^b_\cD$ and $h \sim \phi^\kappa_D$ so that $h^- \le h \le h^+$ and $h^+-h^-=20$ in $\cD$.
In particular, the essential level loops $\cL_1,\dots,\cL_M$ of $h^-$ and $h^+$ coincide.
As before, the number $N$ of essential level loops in $A_{k,\Delta k}$ is at least $c\log\Delta$ with probability at least $1-\Delta^{-\alpha}$. Note that if two of these $N$ loops, say $\cL_i$ and $\cL_j$, have height difference at least 26 in $h^-$ (equivalently, in $h^+$), then $h$ must have a level loop of height 0 mod 6 somewhere in between $\cL_i$ and $\cL_j$. Since an $n$-step simple random walk started from some $x$ reaches $x\pm 26$ with probability at least $1-Ce^{-cn}$, the required bound follows.
\end{proof}

The next corollary roughly says that if a domain contains a box whose size is of the same order as the in-radius of the domain, then, with constant probability, the first essential loop of height 0 modulo 6 inside the box will have height exactly 0 (not merely 0 modulo 6).

\begin{corollary}\label{lem:level_loop2}
There exist $K,c>0$ such that the following holds.
Let $k \ge K$, let $\Delta \ge 2$ and let $D$ be a domain such that $\Lambda_k \subset D  \not\supset \Lambda_{\Delta k}$. Then
\[ \phi^0_D(\text{first essential level loop of height 0 mod 6 inside $\Lambda_k$ has height }0) \ge \frac c{\sqrt{\log\Delta}}. \]
\end{corollary}
\begin{proof}
Let $N$ be the number of essential level loops contained in $\Lambda_k$.
Note that the first $N':=M-N$ loops in $\cL_1,\dots,\cL_M$ are not contained in $\Lambda_k$ and therefore intersect the annulus $A_{k,\Delta k}$.
Given $N$ and $N'$, and on the event that $N \ge 2$, since $(\frac12 H_m)_{m=0}^M$ is a simple symmetric random walk started from 0, the probability that the first essential level loop of height 0 mod 6 inside $\Lambda_k$ exists and has height 0 is at least the probability that a simple random walk of length $N'' := N'+1+\1_{\{N'\text{ even}\}}$ ends at 0. The latter is $2^{-N''} \binom{N''}{N''/2} \ge c/\sqrt{N''}$.
Finally, \cref{lem:level-lines-main} implies that $\phi^0_D(N \ge 2,~N' \le C\log\Delta) \ge 1/2$, as long as $k$ is large enough. Combining these estimates yields the corollary.
\end{proof}

The next corollary shows that small oscillations in the boundary conditions do not effect much the distribution of the height function deep inside the domain. This implies that the convergence of $\mu^{ij}_D$ to $\mu$ extends to a wider class of boundary conditions; see \cref{rmk:sqrtlog}.

\begin{corollary}\label{lem:varying-bc}
For any $\ve>0$, there exist $\delta>0$ such that the following holds. Let $n \ge k \ge 1$, let $D$ be a domain containing $\Lambda_n$ and let $\tau \in \Z^{\partial D}$ be an admissible boundary condition such that $|\tau_v| \le \delta \sqrt{\log \frac nk}$ for all $v \in \partial D$. Then
\[ \dist_{\text{TV}}\left((\phi^\tau_D)_{|\Lambda_k}, (\phi^0_D)_{|\Lambda_k}\right) \le \ve. \]
\end{corollary}
\begin{proof}
Denote $a := \lceil \log \frac nk \rceil$ and $b := \max |\tau_v| \le \delta \sqrt{a}$. It suffices to show that the total variation distance between the restrictions of $\phi^b_D$ and $\phi^{-b}_D$ to $\Lambda_k$ is at most $\ve$, as the two measures $\phi^\tau_D$ and $\phi^0_D$ are stochastically between these by FKG. Let $h^{\pm}$ be sampled from $\phi^{\pm b}_D$, and coupled as follows. First, couple their essential level loops $\cL_1,\dots,\cL_M$ to coincide. Second, couple their heights on these loops, denoted $H^{\pm}_m$, so that their absolute values increase/decrease together (i.e., $H^+_{m+1}-H^+_m = -(H^-_{m+1}-H^-_m)$) until the first time $M'$ where they both reach $H^{\pm}_{M'}=0$, at which point one may couple $h^+$ and $h^-$ so that they coincide in the domain $\cD$ enclosed by $\cL_{M'}$. This coupling shows that $\dist_{\text{TV}}(h^+_{|\Lambda_k},h^-_{|\Lambda_k}) \le \P(\cD\text{ does not contain }\Lambda_k)$. Let $N$ be the number of essential level loops surrounding $\Lambda_k$. By~\eqref{eq:large}, there exists a universal constant $c>0$ and a constant $a_0=a_0(\ve)$ such that $\phi_D^0(N \ge ca) \ge 1-\ve$ if $a \ge a_0$. Standard estimates yield that a simple random walk of length $ca$ started from $\lfloor \delta \sqrt{a} \rfloor$ hits 0 with probability at least $1-\ve$ if $\delta=\delta(\ve)>0$ is chosen small enough. Thus, $\phi_D^0(\Lambda_k \subset \cD) \ge 1-2\ve$ as long as $a \ge a_0$. Finally, note that by decreasing $\delta$ if necessarily, we can ensure that $a<a_0$ implies that $b=0$, in which case the desired statement is trivial.
\end{proof}

\begin{remark}\label{rmk:sqrtlog}
Recall that \cref{thm:colorings-are-locally-mixing} implies that $\mu^{ij}_D$ converges to $\mu$ as $D$ increases to $\Z^2$ along simply connected finite subsets.
Say that a feasible boundary condition $\xi \in \{0,1,2\}^{\partial D}$ has oscillation at most $m$ if there is a height function $\tau$ on $D$ such that $\xi$ equals $\tau$ modulo 3 on $\partial D$ and $|\tau|\le m$ on $\partial D$. Then \cref{lem:varying-bc} implies that $\mu^\xi_D$ converges to $\mu$ as long as $\xi$ has oscillation at most $o(\sqrt{\log (r_D)})$, where $r_D$ is the largest $n$ such that $\Lambda_n \subset D$.
\end{remark}

We end this section with a result about correlation decay for the height function and for the colorings. While the latter decays as a power-law in the in-radius of the domain, the former only decays as a power-law in the logarithm of the in-radius. This shows that a power-law rate in \cref{thm:colorings-are-locally-mixing} is best possible. The argument for the colorings (corresponding to the statement about the height modulo 3) was suggested to us by Ron Peled.

\begin{corollary}\label{cor:correlations}
There exist constants $C,c,\alpha,\beta>0$ such that for any $n \ge 2$ and any domain $D$ having $\Lambda_n \subset D \not\supset \Lambda_{n+1}$,
\[ \frac13 + \frac{c}{n^{\alpha}} \le \phi^0_D(\text{height at origin is 0 modulo 3}) \le \frac13 + \frac{C}{n^\beta} \]
and
\[ \frac{c}{(\log n)^{3/2}} \le \phi^0_D(\text{height at origin is 0}) - \phi^2_D(\text{height at origin is 0}) \le \frac{C}{(\log n)^{3/2}} .\]
\end{corollary}
\begin{proof}
Let $h$ be sampled from $\phi^0_D$ and denote $X := h(0,0)$. Since $X$ is symmetric,
\[ \P(X \in 3\Z) - \tfrac13 = \tfrac23\big[\P(X \in 3\Z) - \tfrac12 \P(X \notin 3\Z)\big] = \tfrac23 \E[\cos(\tfrac{2\pi X}3)] .\]
Recall that the essential level loops $\cL_1,\dots,\cL_M$ were defined as the outermost level loops with $\pm 2$ height increments. Define level loops $\cL'_1,\dots,\cL'_N$ similarly as the outermost level loops with $\pm 1$ increments (so that the essential level loops are a subset of these). Note that, given $N$, $X$ is distributed as the sum of $N$ independent uniform $\pm 1$ variables. Thus, $$\E[\cos(\theta X) \mid N] = \Re \E[e^{i\theta X} \mid N] = \cos(\theta)^N.$$ As $N$ is even (since $X$ is), plugging in $\theta = \frac{2\pi}3$, we get that $\E[\cos(\tfrac{2\pi X}3) \mid N] = 2^{-N}$. Thus,
\[ \P(X \in 3\Z) = \tfrac13 + \tfrac23 \E[2^{-N}] .\]
It remains to bound $\E[2^{-N}]$ from above and below. By tuning the constants in the statement, we may assume that $n$ is sufficiently large.
To obtain the stated upper bound, it suffices to show that $\P(N \le c\log n) \le n^{-\zeta}$ for some universal constants $c,\zeta>0$. Since $N \ge M$, this follows from~\eqref{eq:large}.
For the lower bound, it suffices to show that $\P(N \le C\log n) \ge \tfrac14$ for some universal constant $C>0$. It follows from~\eqref{eq:small} that $\P(M \le C\log n) \ge \tfrac12$. To obtain the bound for $N$, note that, given $N$, $M$ is distributed as $\max \{m : T_1+\cdots+T_m \le N\}$, where $\{T_i\}$ are independent copies of $T:=\min\{ k : |\xi_1+\dots+\xi_k|=2 \}$, where $\{\xi_i\}$ are i.i.d.\ uniform $\pm 1$ variables. By Markov's inequality, $\P(M<cN) \le \P(T_1+\cdots+T_{\lceil cN \rceil} > N) \le \frac14$ for some universal constant $c>0$. Thus, $\P(N \ge (C/c)\log n) \le \frac34$, as required.

We now turn to the second inequality in the corollary.
Observe that the quantity in question is the same as $\P(X=0)-\P(X=2)$. Conditioning on $N$, we have
\begin{align*}
\P(X=0 \mid N)-\P(X=2 \mid N) &= 2^{-N}\binom{N}{N/2} - 2^{-N}\binom{N}{N/2+1} \\&= \frac{2^{-N}}{N/2+1}\binom{N}{N/2} = \Theta(N^{-3/2}) .
\end{align*}
Using the same bounds on $N$ as before yields the required inequality.
\end{proof}

\subsubsection{Proof of Proposition~\ref{lem:level-lines-main}}

We begin with two lemmas.
The first lemma asserts that if we look at the height function inside a domain containing some box $\Lambda$, then, with high probability, we will find a level loop of absolute height at least 2 surrounding a slightly smaller box. More precisely, the loop will be between the boundary of the domain (which could be far away $\Lambda$) and the boundary of a box whose size is at most a geometric number of scales smaller than $\Lambda$.

\begin{lemma}\label{lem:inside_quick}
There exists a constant $c>0$ such that the following holds. Let $k \ge 1$, let $D$ be a domain containing $\Lambda_k$ and let $(B,\kappa)$ be $|h|$-adapted boundary conditions  with $B_{\rm sym} = \partial D$ and $B_{\rm pos} = \emptyset$. 
Let $\cE_i$ be the event that the annulus $A_{2^{-i}k, 2^{-i+1}k}$ contains a level loop of $|h|\ge 2$ surrounding the origin. Then, for all $n \ge 1$ with $2^{-n}k \ge 100$,
\[ \phi^\kappa_D (\cE_1 \cup \dots \cup \cE_n) \ge 1 - e^{-cn}. \]
\end{lemma}

\begin{proof}
Observe that $\cE_i$ is measurable with respect to $|h|_{|A_i}$, where $A_i$ is the annulus $A_{2^{-i}k, 2^{-i+1}k}$.
It suffices to show that
\[ \phi^\kappa_D (\cE_n \mid \1_{\cE_1},\dots,\1_{\cE_{n-1}}) \ge c \qquad\text{almost surely}. \]
In fact, we prove the stronger statement that
\[ \phi^\kappa_D (\cE_n \mid |h|_{|D \setminus D_n}) \ge c \qquad\text{almost surely}, \]
where $D_n$ is the domain $\Lambda^{\text{e}}_{2^{-n+1}k-1}$. In words, we first explore the absolute value from the boundary of $D$ inward until we reach the outer boundary of $A_n$ and their even neighbors inside $A_n$, and then, regardless of what this exploration reveals, the conditional probability of $\cE_n$ is at least some universal constant $c$.

Toward showing this, we first argue that, when sampling from $\phi^\kappa_D$, the conditional distribution of $|h|_{|D_n}$ given $|h|_{|(D \setminus D_n) \cup \partial D_n}$ is almost surely stochastically larger than $\phi^0_{D_n}(|h| \in \cdot)$. Indeed, the domain Markov property and the assumption that the boundary condition has $B_{\rm pos} = \emptyset$ imply that the former conditional distribution is $\phi^{\kappa_n}_{D_n}(|h| \in \cdot)$, where $\kappa_n$ is the boundary condition on $\partial D_n$ that equals $\{|h_v|,-|h_v|\}$ at each $v \in \partial D_n$. The FKG for absolute value (\cref{prop:FKG_modh}) now implies that $\phi^{\kappa_n}_{D_n}(|h| \in \cdot)$ stochastically dominates $\phi^0_{D_n}(|h| \in \cdot)$.

Now observe that $\cE_n$ is increasing in $|h|$ so that the above stochastic domination implies that $\phi^\kappa_D(\cE_n \mid |h|_{|D \setminus D_n}) \ge \phi_{D_n}^0(\cE_n)$ almost surely.
Finally, the uniform lower bound on $\phi_{D_n}^0(\cE_n)$ follows from \eqref{eq:ha2} and the standard trick of using FKG (for $|h|$) to glue together crossings of four rectangles into a loop.
\end{proof}

\begin{corollary}\label{cor:inside_quick}
There exists a constant $c>0$ such that for all $k \ge 1$, all domains $D \supset \Lambda_k$ and all $n \ge 1$ with $2^{-n}k \ge 100$, 
$$
\phi_D^0(\Lambda_{2^{-n}k}\text{ is surrounded by a level loop of $|h|=2$ in $D$}) \ge 1-e^{-cn}.
$$
\end{corollary}
\begin{proof}
\cref{lem:inside_quick} implies that with probability at least $1-e^{-cn}$ there is a $\times$-loop of $|h|\ge 2$ surrounding $\Lambda_{2^{-n}k}$ in $D$. Since the boundary conditions put height 0 on $\partial D$, on the former event, there must be a level loop of $|h|=2$ surrounding $\Lambda_{2^{-n}k}$ in $D$.
\end{proof}

The next lemma essentially shows that a level loop of height 2 intersecting a box cannot have long `tentacles' going very far away from the box.
A \textbf{$\times$-crossing} of an annulus inside a domain $D$ is a $\times$-path in $D$ connecting the inside and the outside boundaries of the annulus.

\begin{lemma}\label{lem:tentacles}
There exists a constant $c>0$ such that the following holds. Let $D$ be a domain and let $k \ge 100$.
For $i \ge 1$, let $\cC_i$ be the event that there exists a $\times$-crossing of height $2$ of the annulus $A_{2^{i-1} k, 2^ik}$ inside $D$.
Then, for all $n \ge 1$,
$$\phi_D^0(\cC_1 \cap \dots \cap \cC_n) \le e^{-cn}.$$
\end{lemma}

We remark that no assumption on the location of the domain is made. In particular, it may intersect some of the annuli without containing them.

\begin{proof}
Let $A_i$ be the slightly thinner \emph{even} annulus $(\Lambda^{\text{e}}_{2^ik-10} \setminus \Lambda^{\text{e}}_{2^{i-1}k+10}) \cup \partial \Lambda^{\text{e}}_{2^{i-1}k+10}$. With a slight abuse of notation, we redefine $\cC_i$ to be the event that there exists a $\times$-crossing of height 2 of this thinner annulus inside $D$, noting that this increases the event so that it suffices to prove the stated inequality with this new definition.

Denote $D_i := A_i \cap D$ and $B := \bigcup_{i \ge 1} \partial D_i$, and let $B_0 := B \setminus \partial D$ be the portions of the boundaries of the even annuli in $D \setminus \partial D$. Consider the boundary condition $(B,\kappa)$ in which $\kappa$ equals $\{0\}$ on $\partial D$ and $\{2\}$ on $B_0$.
We first show that
\[ \phi^0_D(\cC_1 \cap \dots \cap \cC_n) \le \phi_D^{B,\kappa} (\cC_1 \cap \dots \cap \cC_n) .\]
Let $\cC'_i$ be the event that there exists a $\times$-crossing of height $0$ of the annulus $A_i$ inside $D$.
Note that, by symmetry,
\[ \phi^0_D(\cC_1 \cap \dots \cap \cC_n) = \phi^{-2}_D(\cC'_1 \cap \dots \cap \cC'_n) = \phi^2_D(\cC'_1 \cap \dots \cap \cC'_n) \]
and, similarly,
\[ \phi^{B,\kappa}_D(\cC_1 \cap \dots \cap \cC_n) = \phi^{B,\kappa'}_D(\cC'_1 \cap \dots \cap \cC'_n) ,\]
where $\kappa'$ equals $\{2\}$ on $\partial D$ and $\{0\}$ on $B_0$. Thus, it suffices to show that
\[ \phi^2_D(\cC'_1 \cap \dots \cap \cC'_n) \le \phi_D^{B,\kappa'} (\cC'_1 \cap \dots \cap \cC'_n) .\]
Since $\cC'_i$ is a decreasing event in the absolute value of $h$, this follows from the FKG for absolute value (\cref{prop:FKG_modh}), where the constant boundary condition 2 is viewed as the boundary condition $(B,\kappa'')$ in which $\kappa''$ equals $\{2\}$ on $\partial D$ and $\Z$ on $B_0$ (the latter is equivalent to having no condition on the height on $B_0$, i.e., vertices on $B_0$ are free to take any value). Here, both boundary conditions have the same partition of $B$ into $B_{\rm pos}=\partial D$ and $B_{\rm sym}=B_0$.

Observe that the boundary condition $(B,\kappa)$ decouples the events $\cC_1,\dots,\cC_n$. Thus,
\[ \phi^0_D(\cC_1 \cap \dots \cap \cC_n) \le \phi_D^{B,\kappa} (\cC_1 \cap \dots \cap \cC_n) = \phi_D^{B,\kappa}(\cC_1) \cdots \phi_D^{B,\kappa}(\cC_n) = \phi_{D_1}^{\kappa_1}(\cC_1) \cdots \phi_{D_n}^{\kappa_n}(\cC_n) ,\]
where $\kappa_i$ equals $\{0\}$ on $\partial D_i \cap \partial D$ and $\{2\}$ on $\partial D_i \setminus \partial D$. It remains to show that
\[ \phi_{D_i}^{\kappa_i}(\cC_i) \le 1-\alpha \]
for some universal constant $\alpha>0$.
The case when $D_i=A_i$ (i.e., when the domain $D$ contains the annulus $A_i$) is more straightforward than the case $D_i \subsetneq A_i$.
In order to treat the latter case in a similar way as the former (and simultaneously), it is convenient to work in the entire annulus $A_i$ rather than in the portion $D_i$ of it. To this end, with a slight abuse of notation, we extend $\kappa_i$ to $B_i=\partial D_i \cup \partial A_i$ by defining it to be $\{2\}$ on $\partial A_i \setminus \partial D_i$. Then, by the domain Markov property, $\phi_{D_i}^{\kappa_i}(\cC_i) = \phi_{A_i}^{B_i,\kappa_i}(\cC_i)$.
Now consider the event $\bar{\cC}_i$ that there exists a $\times$-crossing of height $2$ of the annulus $A_i$ (not necessarily inside $D$). Clearly, $\cC_i \subset \bar{\cC}_i$ so that $\phi_{A_i}^{B_i,\kappa_i}(\cC_i) \le \phi_{A_i}^{B_i,\kappa_i}(\bar{\cC}_i)$.
Let $\cE_i$ be the event that there exists a $*$-loop of height equal to 0 surrounding the origin inside $A_i$ (a $*$-path is a sequence of vertices in $\Z^2$ with consecutive vertices at nearest-neighbor distance equal to 2).
By duality, the events $\bar{\cC}_i$ and $\cE_i$ are disjoint Indeed, by~\cite[Lemma 2.4, first item]{duminil2019logarithmic}, a $\times$-cluster of height at least $2$ is blocked by a $*$-loop of height at most 0. Since, however, the boundary conditions on $\partial A_i$ are at least 0 everywhere (namely, 0 or~2), the latter is equivalent to having a $*$-loop of height exactly equal to 0 in the annulus. Thus, $\phi_{A_i}^{B_i,\kappa_i}(\bar{\cC}_i) \le 1 - \phi_{A_i}^{B_i,\kappa_i}(\cE_i)$. So far, we have shown that
\[ \phi_{D_i}^{\kappa_i}(\cC_i) \le 1 - \phi_{A_i}^{B_i,\kappa_i}(\cE_i) .\]

It remains to show that $\phi_{A_i}^{B_i,\kappa_i}(\cE_i)$ is uniformly bounded below.
To this end, we first argue that $\phi_{A_i}^{B_i,\kappa_i}(\cE_i) \ge \phi_{A_i}^2(\cE_i)$.
Indeed, since the event $\cE_i$ is decreasing in $|h|$, by viewing the constant 2 boundary condition as equaling $\Z$ on $B_i \setminus \partial A_i$ (i.e., taking free boundary conditions there), this follows from the FKG for absolute value.
Finally, the estimate $\phi_{A_i}^2(\cE_i) \ge \alpha$ follows from \eqref{eq:ha2} and the standard trick of using FKG (again, for $|h|$) to glue together crossings of four rectangles into a loop.
\end{proof}

We are now ready to prove \cref{lem:level-lines-main}.

\begin{proof}[Proof of \cref{lem:level-lines-main}]
For $j>i \ge 0$, let $N_{i,j}$ be the number of essential level loops contained in $A_{2^ik,2^jk}$ and let $N'_{i,j} \ge N_{i,j}$ be the number of essential level loops intersecting $A_{2^ik,2^jk}$. We are interested in lower bounding $N=N_{0,a}$ and upper bounding $N'=N'_{0,a}$. For the reader's convenience, we recall the precise statements in the proof below.

\medbreak
\noindent\textbf{Upper bound:}
We show the upper bound~\eqref{eq:small} on $N'$, namely, that for some universal constants $C,c>0$, we have
\begin{equation}\label{eq:small2}
\phi_D^0(N' \ge n) \le e^{-cn \log(\frac na \wedge k)} \qquad\text{for any }n \ge Ca .
\end{equation}
The proof follows a strategy similar to that of~\cite[Proposition 4.8]{duminil2019logarithmic} with some simplifications.
Let $D_m$ be the domain enclosed by $\cL_m$ and recall that $H_m$ is the height on $\cL_m$.
Consider the line $(k,2^ak] \times \{0\}$ crossing the annulus $A_{k,2^ak}$.
Let $\ell_m \in (k,2^ak]$ denote the distance to the origin of the right-most intersection point $(\ell_m,0)$ of $\cL_m$ with this line. Fix a large $r$, and let $\cB_m$ be the event that $\ell_m-\ell_{m+1} \le 2^{-r}\ell_m$, or equivalently, $\ell_{m+1} \ge (1-2^{-r})\ell_m$. In particular, on this event, the loops $\cL_m$ and $\cL_{m+1}$ are relatively close to one another.
Specifically, if $\cB_m$ occurs, then the loop $\cL_{m+1}$ creates a crossing of $A_{2^{-r}\ell_m,\ell_m} ((\ell_m,0)) \cap D_m$ by a $\times$-path of height $H_{m+1}$. Thus, by \cref{lem:tentacles},
\[ \phi_D^0(\cB_m \mid \cL_1,\dots,\cL_m) \le e^{-dr} ,\]
where $d>0$ is a universal constant, as long as $2^{-r}\ell_m \ge 100$.

Since $\cB_m$ occurs unless $\log \ell_m - \log \ell_{m+1} > -\log(1-2^{-r}) \ge 2^{-r}$, it is straightforward that all but at most $2^ra \log 2$ of the essential level loops $\cL_m$ intersecting $A_{k,2^ak}$ must trigger the corresponding event $\cB_m$. Thus, by the above bound on the conditional probability of $\cB_m$ and a union bound on those $m$ for which $\cB_m$ occurs, we have
\begin{equation*}
\phi_D^0 \left(N' = n \right) \le 2^n e^{-dr(n-2^ra \log 2)} .
\end{equation*}
Choosing now $r$ to be the minimum between $\log_2 (\frac n{2a})$ and $\log_2(\frac{k}{100})$ (so as to ensure that $2^{-r}\ell_m \ge 100$ always holds) and summing over $n$ yields the required bound~\eqref{eq:small2}.

\medbreak
\noindent\textbf{Lower bound:}
We now turn to the lower bound~\eqref{eq:large} on $N $, namely, that for some universal constants $C,c>0$, if $D$ contains $\Lambda_{2^ak}$, then
\[ \phi_D^0(N \le ca) \le Ce^{-ca} .\]
The main issue to overcome is the case when $D$ is much further away from the annulus in question. We need to get around the fact that the outermost essential level loop which intersects $\Lambda_{2^ak}$ is not too irregular. We divide the proof into three steps.

\smallskip
\noindent\textbf{Step 1:}
We first establish a similar bound on the number of essential level loops surrounding $\Lambda_k$. Namely,
\begin{equation}\label{eq:large2}
\phi_D^0(N_{0,\infty} \le ca) \le e^{-ca}.
\end{equation}
For $0 \le m < N_{0,\infty}$, let $S_m$ denote the largest integer $i \ge 0$ such that $\cL_m$ surrounds $\Lambda_{2^ik}$. We think of $S_m$ as the scale at which we fully discover $\cL_m$. Set $S_m=-1$ for $m \ge N_{0,\infty}$.
Note that $(S_m)_{m=0}^\infty$ is a decreasing sequence with $S_0 \ge a$ and $S_{N_{0,\infty}-1} \ge 0$.
By \cref{cor:inside_quick}, for any $m \ge 0$, the conditional distribution of $S_m-S_{m+1}$ given $\cL_1,\dots,\cL_m$ is almost surely stochastically dominated by a geometric random variable of some universal parameter $p$.
Thus,
\[ N_{0,\infty}\text{ stochastically dominates }1+\max \{ m \ge 0 : \xi_1+\cdots+\xi_m \le a\} ,\]
where $\{\xi_i\}$ are independent geometric random variables with parameter $p$. It now follows from a standard Chernoff bound that there exists a universal constant $c>0$ such that
\[ \phi_D^0(N_{0,\infty} \le ca) \le \P(\xi_1+\cdots+\xi_{\lfloor ca \rfloor} > a) \le e^{-ca} .\]

\smallskip
\noindent\textbf{Step 2:}
Next, we establish the lower bound on $N$ in the case when the domain $D$ does not contain $\Lambda_{2^{a+1}k}$. In this case, any essential level loop not contained in $\Lambda_{2^ak}$ must intersect $A_{2^ak,2^{a+1}k}$, so that $N=N_{0,\infty} - N'_{a,a+1}$. Thus,
\[ \phi_D^0(N \le \tfrac{ca}2) \le \phi_D^0(N_{0,\infty} \le ca) + \phi_D^0(N'_{a,a+1} \ge \tfrac{ca}2) \le Ce^{-ca}, \]
where the last inequality follows from~\eqref{eq:large2} and~\eqref{eq:small2}. We note for use in the next step that this inequality can be restated (by applying it with a different choice of $k$ and $a$, and readjusting the constants appropriately) as, for any $a_2 \ge a_1 \ge 0$ and any domain $D$ such that $\Lambda_{2^{a_2}k} \subset D \not\supset \Lambda_{2^{a_2+1}k}$,
\begin{equation}
\phi_D^0(N_{a_1,a_2} \le c(a_2-a_1)) \le Ce^{-c(a_2-a_1)} .\label{eq:simpler_2}
\end{equation}

\smallskip
\noindent\textbf{Step 3:}
Finally, we are ready to establish the bound for an arbitrary domain $D$ containing $\Lambda_{2^ak}$.
Let $I$ be the smallest integer $i>a/2$ such that there exists an essential level loop surrounding $\Lambda_{2^{a/2}k}$ and intersecting $\Lambda_{2^{i+1}k}$ (recall that we also consider the boundary of the domain as an essential level loop, so that $I$ is always finite).
It suffices to show that
\[ \phi_D^0(N \le ca,~I=i) \le Ce^{-ci} \qquad\text{for all }i>a/2 .\]
Fix $i>a/2$ and suppose that the event $\{I=i\}$ occurs.
Let $M_0$ denote the smallest index $0 \le m \le M$ such that the loop $\cL_m$ intersects $\Lambda_{2^{i+1}k}$, and note that this loop necessarily surrounds $\Lambda_{2^ik}$.
Let $D_0$ be the domain enclosed by $\cL_{M_0}$, and note  that $D_0$ contains $\Lambda_{2^ik}$ but not $\Lambda_{2^{i+1}k}$.
Note also that we may explore $\cL_1,\dots,\cL_{M_0}$ from the outside, so that the height function inside $D_0$ is conditionally independent by the domain Markov property.
We now consider two cases: if $i \le a$, then $N \ge N_{0,i}$, so that~\eqref{eq:simpler_2} yields the required bound, and if $i>a$, then $N_{a/2,i}=0$, so that~\eqref{eq:simpler_2} again yields the required bound.
\end{proof}

\subsection{Proof of Theorem~\ref{thm:colorings-are-locally-mixing}}
\label{sec:proof-main-thm}

Our goal is to show that the family $\cM$ of all pairs $(\mu^{ij}_D,D)$, with $D$ finite and simply connected and $i,j \in \{0,1,2\}$ distinct, is locally mixing with a power-law rate function. To this end, we fix $n \ge 1$, distinct $i,j \in \{0,1,2\}$, distinct $i',j' \in \{0,1,2\}$ and two finite and simply connected sets $D,D' \subset \Z^2$ containing $\Lambda_n$, and we aim to construct a coupling between $f \sim \mu^{ij}_D$ and $f' \sim \mu^{i'j'}_{D'}$ with the desired properties (recall \cref{def:local-mixing}). We first handle the case when $D$ and $D'$ are domains and $i=i'=0$ (in which case, $j$ and $j'$ are irrelevant), and later explain the general case. We further assume without loss of generality that one of the domains, say $D'$, is $\Lambda_n^{\text{e}}$, in which case the other domain necessarily contains it (the required coupling is then easily obtained from two such couplings).

In order not to introduce cumbersome notation, we begin by describing the coupling ``as seen from'' a single vertex $v$. The construction will involve an iteration of a 3-step procedure in which we alternate between exploring \textbf{(1)} only $f$, \textbf{(2)} both $f$ and $f'$ simultaneously, and \textbf{(3)} only $f'$. At the end of iteration $i$, we will have defined a sequence $L_0,L'_0,L_1,L'_1,\dots,L_i,L'_i$ of nested even $\times$-loops surrounding $v$ (some loops may coincide or partially overlap), where $L_0,\dots,L_i$ are color-0 loops of $f$ and $L'_0,\dots,L'_i$ are color-0 loops of $f'$, and we will have fully explored $f$ and $f'$ outside of $L_i$ and $L'_i$, respectively, but not at all inside the domains $D_i$ and $D'_i$ enclosed by them; see \cref{fig:loops}. In particular, the domain Markov property will imply that at the end of iteration $i$, the conditional distributions of $f_{|D_i}$ and $f'_{|D'_i}$ will be $\mu^0_{D_i}$ and $\mu^0_{D'_i}$, respectively.

\begin{figure}
 \centering
 \includegraphics[scale=0.4]{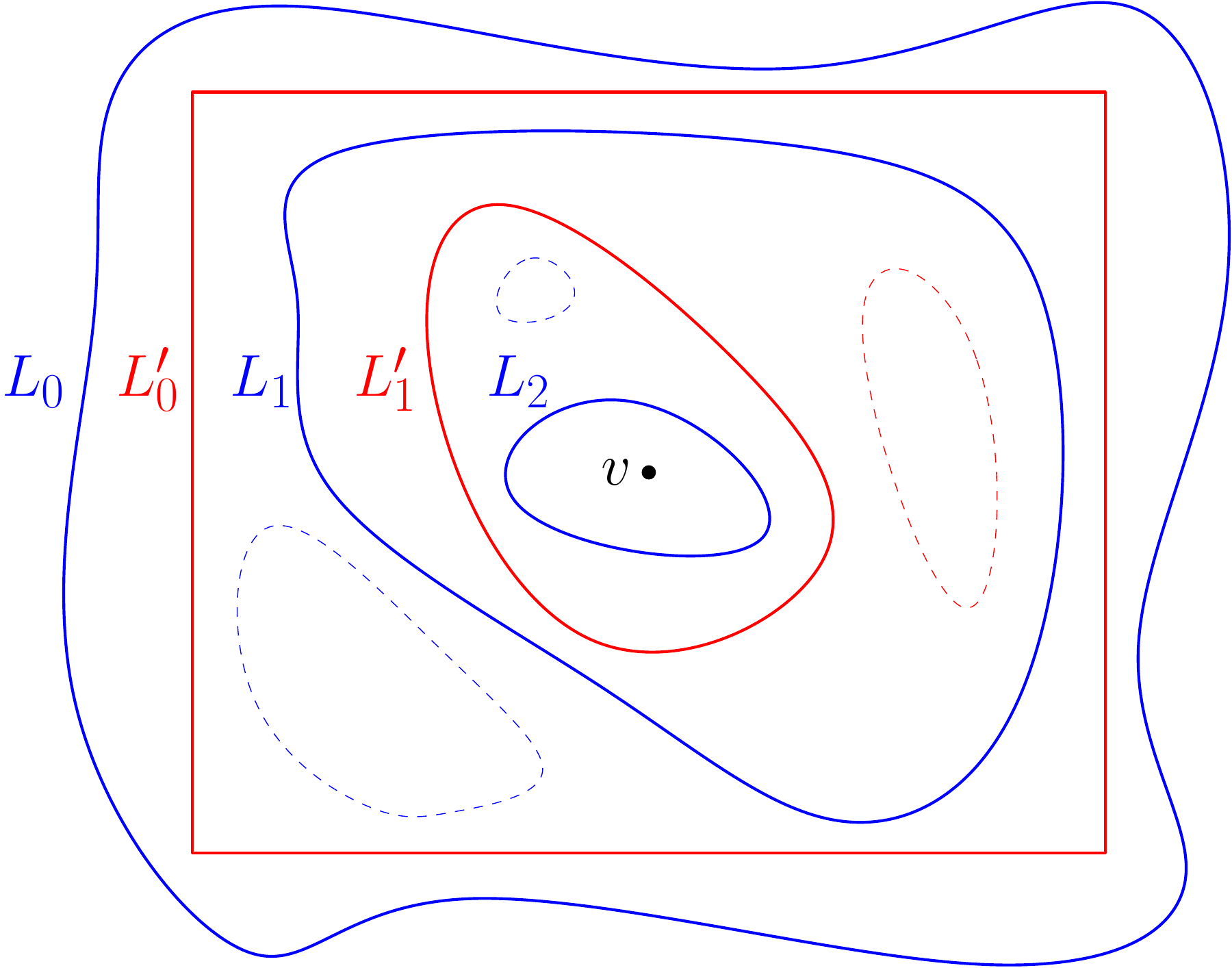}
 \caption{The sequence of nested loops $L_0,L'_0,L_1,L'_1,\dots$ surrounding $v$.}
 \label{fig:loops}
\end{figure}

Initially, we set $L_0:=\partial D$ and $L'_0:=\partial D'$, and note that $L_0$ surrounds $L'_0$ by our assumption on the domains. This is all that is done in iteration 0. Suppose that we have completed iteration $i$ and that the conditional distributions of $f_{|D_i}$ and $f'_{|D'_i}$ are $\mu^0_{D_i}$ and $\mu^0_{D'_i}$. Let us now explain how to define $L_{i+1}$ and $L'_{i+1}$. At this point, we switch to the height function representation, noting that we may view $f_{|D_i}$ as a height function $h_i$ for which $L_i$ is a level loop of height 0, and similarly, we may view $f'_{|D'_i}$ as a height function $h'_i$ for which $L'_i$ is a level loop of height 0. Thus, $h_i \sim \phi^0_{D_i}$ and $h'_i\sim\phi^0_{D'_i}$.
We now use this representation in order to describe the 3-step exploration procedure of $f$ and $f'$ in iteration $i+1$. For readability, we write $h=h_i$ and $h'=h'_i$ below.

\textbf{Step 1:} We explore the absolute value $|h|$ of the height function $h$ (independently of $h'$) up to the loop $L'_i$ (i.e., on $D_i \setminus D'_i$ and on the boundary of $D'_i$), thereby revealing a random $|h|$-adapted boundary condition for $h$ on the domain $D'_i$. By the domain Markov property, the conditional law of $h$ in $D'_i$ is that of a uniform homomorphism height function in $D'_i$ with this boundary condition (which has $B_{\rm sym}=L'_i$ and $B_{\rm pos}=\emptyset$).

\textbf{Step 2:} Since $h'$ is zero on $L'_i$, the FKG for absolute value implies that $|h|_{|D'_i}$ stochastically dominates $|h'|_{|D'_i}$ (given the exploration in the first step).
Consequently, we can construct a coupling between $|h|$ and $|h'|$ inside $D'_i$ by simultaneously exploring and revealing the values of both $|h|$ and $|h'|$, vertex by vertex, starting from the boundary of $D'_i$ inwards, ensuring along the way that $|h| \ge |h'|$ (this type of exploration is standard; see, e.g., \cite{van1994disagreement}). We explore $|h|$ and $|h'|$ in this way until we discover the outermost level loop of height 0 mod 6 for $h$ surrounding $v$ and inside $D'_i$; this loop is $L_{i+1}$ (note that it is a color-0 even $\times$-loop for $f$).

\textbf{Step 3:} The domain Markov property implies that, at this point, the conditional distributions of $|h|_{|D_{i+1}}$ and $|h'|_{|D_{i+1}}$ are that of the absolute values of uniform homomorphism height functions in $D_{i+1}$ with the corresponding $|h|$-adapted boundary conditions that have been revealed by the exploration in the previous step.
If these boundary conditions happen to be identical, then we jointly sample $h$ and $h'$ in the domain $D_{i+1}$ according to a single sample from their common distribution, and we stop the iterative procedure; in this case, we say that the $i+1$ iteration resulted in a \textbf{successful coupling for $v$}.
Otherwise, we continue exploring $|h'|$ alone (independently of $h$) until we discover the outermost level loop of height 0 mod 6 for $h'$ surrounding $v$ and inside $D_{i+1}$; this loop is $L'_{i+1}$. At this point, we have coupled the absolute values of the height functions up to $L_{i+1}$ and $L'_{i+1}$. Finally, we complete the coupling $h$ and $h'$ by coupling their signs in any manner (e.g., independently).

This completes the description of iteration $i+1$, at the end of which we have explored $f$ up to $L_{i+1}$ and $f'$ up to $L'_{i+1}$, so that by the domain Markov property, the conditional distributions of $f_{|D_{i+1}}$ and $f'_{|D'_{i+1}}$ are indeed $\mu^0_{D_{i+1}}$ and $\mu^0_{D'_{i+1}}$, as claimed. If at some iteration~$i$ we cannot find one of the loops we are looking for (i.e., $L_i$ or $L'_i$ does not exist), then we stop the iterative procedure and say that the \textbf{coupling has failed for~$v$}.
We emphasize that in iteration~$i$, when we discover the loops $L_i$ and $L'_i$, the heights of $h$ and $h'$ on these loops are some (different) multiples of 6, but that in the next iteration we then shift each of the two height functions by the corresponding amount so that the height again becomes 0 for both before we start looking for the loops $L_{i+1}$ and $L'_{i+1}$.

The actual coupling between $f$ and $f'$ does not treat $v$ as a distinguished vertex, but rather attempts to couple all vertices in $\Lambda_n$ in parallel. Specifically, whenever we tried to find the outermost level loop of height 0 mod 6 surrounding $v$, we instead find \emph{all} outermost level loops of height 0 mod 6 (with no specific target vertex). The collection of all such loops is still explorable from the outside. Thus, iteration $i$ generates for us two collections of loops $\{L_{i,j}\}_j$ and $\{L'_{i,j,j'}\}_{j,j'}$ (with $L'_{i,j,j'}$ nested inside $L_{i,j}$), and in iteration $i+1$ we recursively repeat this procedure inside each of these loops. This completes the description of the coupling between $f$ and $f'$.

\medskip

We now turn to show that the constructed coupling has the required properties.
We first check that $f_{|D \setminus \Lambda_n}$ and $f'$ are independent. Indeed, in the first step of the first iteration of the construction, we explore $f$ independently of $f'$ up to $\partial D'$ and thereby reveal $f_{|(D \setminus D') \cup \partial D'}$. Since $D' \setminus \partial D' \subset \Lambda_n$ and since, conditionally on this exploration, the law of $f'$ is still $\mu^0_{D'}$, we see that $f'$ is independent of $f_{|D \setminus \Lambda_n}$.

We now turn to the main issue at hand, namely, to show that the constructed coupling has a good chance to successfully couple any given vertex. Precisely, we need to show that, under this coupling, $\P(f(v) \neq f'(v)) \le C(n-k)^{-\alpha}$ for some $C,\alpha>0$ and any $0 \le k \le n$ and $v \in \Lambda_k$. Fix such a $k$ and $v$. By construction, $f(v)$ and $f'(v)$ are equal unless the coupling fails for $v$. Thus, it suffices to show that
\begin{equation}\label{eq:failure-bound}
\P(\text{coupling fails for }v) \le C(n-k)^{-\alpha} .
\end{equation}
We continue to use the loops $L_i$ and $L'_i$ as defined above with respect to $v$. 
Let $\cF_i$ denote the $\sigma$-algebra generated by $L_0,L'_0,\dots,L_i,L'_i$ and $f_{|(D \setminus D_i) \cup L_i}$ and $f'_{|(D' \setminus D'_i) \cup L'_i}$. Note that $\cF_i$ represents the information revealed at the end of iteration $i$.

Let us first show that it is unlikely that the coupling fails for $v$ after few iterations. Precisely, we claim that for some constants $C,c,\alpha>0$, we have
\begin{equation}\label{eq:many-iterations}
\P(\text{coupling fails for $v$ before iteration }c \log(n-k)) \le C(n-k)^{-\alpha} .
\end{equation}
To this end, let $S_i$ denote the largest $j \ge 0$ such that $L'_i$ surrounds $\Lambda_{2^j}(v)$, and set $S_i=-1$ if $L'_i$ does not exist.
Note that $S_0 \ge S_1 \ge S_2 \ge \cdots$ since the loops are nested, and that $S_0 \ge \lfloor \log_2(n-k) \rfloor$ since $\Lambda_{n-k}(v) \subset \Lambda_n \subset D'$. Note also that if the coupling fails for $v$ before iteration $i$, then $S_i=-1$. Thus, it suffices to bound $\P(S_m=-1)$, where $m := \lceil c\log(n-k) \rceil$. We claim that, for any $i \ge 0$, conditioned on $\cF_i$, the difference $S_i-S_{i+1}$ is almost surely stochastically dominated by a random variable $T$ having exponential tails. Indeed, if $S_i=s$ and $S_i-S_{i+1} \ge 2t+1$, then $t \le s/2$ and either the annulus $A_{2^{s-t},2^s}(v)$ contains no loop of height 0 mod 6 surrounding $v$ for $h_i$ or the annulus $A_{2^{s-2t},2^{s-t}}(v)$ contains no loop of height 0 mod 6 surrounding $v$ for $h'_i$.
\cref{lem:monochromatic_loop} implies that, given $\cF_i$, each of these events has probability at most $Ce^{-ct}$ for some universal constants $C,c>0$.
Therefore, letting $\{T_j\}_j$ be independent copies of $T$,
\[ \P(S_m=-1) \le \P(T_1+\cdots+T_m \ge \lfloor \log_2(n-k) \rfloor) \le C(n-k)^{-\alpha} ,\]
where the second inequality follows from a standard Chernoff bound for i.i.d.\ random variables with exponential tails by choosing $c$ small enough.

Now that we know that the coupling does not fail for $v$ before order $\log(n-k)$ iterations, we aim to show that in each such iteration there is a constant probability of a successful coupling for $v$. It will be helpful to consider two consecutive iterations at a time. Thus, we aim to show that, for some constants $K,c>0$, for any even $i \ge 0$, conditioned on $\cF_i$, on the event that $L'_i$ is at distance at least $K$ from $v$,  with probability at least $c$, the $i+2$ iteration results in a successful coupling for $v$.
This yields that, for all $i \ge 0$,
\[ \P(\text{coupling neither succeeds nor fails for $v$ before iteration }i) \le C(1-c)^i .\]
Together with~\eqref{eq:many-iterations}, this gives the required bound~\eqref{eq:failure-bound}.
The reason for considering two consecutive iterations is that, given $\cF_i$, it may happen that $L'_i$ is deep inside $L_i$. We use the first of the two iterations to gain some control on this, showing that, regardless of the relative geometry of $L_i$ and $L'_i$, there is a constant probability that $L'_{i+1}$ is not far from $L_{i+1}$. In the next iteration, conditioning on $\cF_{i+1}$, we may then assume that we are on this good event, in which case we will be able to show that there is a constant probability that $L_{i+2}$ is a color-0 loop for both $f$ and $f'$ (in fact, a loop of height 0 for both $h_{i+1}$ and $h'_{i+1}$), resulting in a successful coupling for $v$. We now make this precise.

Condition on $\cF_i$ and consider the loops $L_i$ and $L'_i$.
By what we have showed above about $S_i-S_{i+1}$, with probability at least $1/2$, we have that
\begin{equation}\label{eq:controlled-loops}
\Lambda_{2^{s-a}}(v) \subset D'_{i+1} \subset D_{i+1} \subset D'_i \not\supset \Lambda_{2^{s+1}}(v)
\end{equation}
for some universal constant $a$, where $s:=S_i$.
Now condition on $\cF_{i+1}$ and assume that~\eqref{eq:controlled-loops} occurs. Then \cref{lem:level_loop2} implies that $L_{i+2}$ exists and is a level loop of height 0 for $h_{i+1}$ with probability at least $c$ for some universal constant $c>0$, as long as $2^{s-a}$ is larger than some universal constant $K'$ (which is ensured by choosing $K=2^{a+1}K'$). Now observe that the domination maintained in the second step of the construction of the coupling implies that $L_{i+2}$ must also be a level loop of height 0 for $h'_{i+1}$, implying that iteration $i+2$ resulted in a successful coupling for $v$. Thus, there is probability at least $c/2$ that the $i+2$ iteration results in a successful coupling for $v$. This finishes the proof that the constructed coupling has the two properties required by the local mixing condition.

We have shown above that the family $\cM' \subset \cM$ consisting of all pairs $(\mu^0_D,D)$, where $D$ is a domain, is locally mixing with a power-law rate function. It remains to explain that $\cM$ is locally mixing with such a rate function.
Suppose that $f \sim \mu^{ij}_D$ for general $D,i,j$. We may still assume as before that $D'$ is the domain $\Lambda^{\text{e}}_n$ and that $i'=0$. The above proof applies to this situation as is, with the only difference being that $L_0$ is no longer a color-0 loop for $f$ so that when we appeal to \cref{lem:monochromatic_loop} in the first iteration (i.e., when arguing that $L_1$ is discovered quickly), we need to use the full strength of the corollary (the moreover part). In fact, this shows more, namely, that the larger class $\cM'' \supset \cM$ consisting of all pairs $(\mu^\xi_D,D)$ with $D$ finite and simply connected and $\xi$ a feasible boundary condition with bounded oscillation (as in the sense of \cref{rmk:sqrtlog}) is also locally mixing with a power-law rate function.
\qed

\section{F{\o}lner independence implies local mixing}
\label{sec:FI}

In this section, we complete the proof of \cref{prop:locally-mixing-and-FI}, by showing that any translation-invariant measure $\mu$ on $\cA^{\Z^d}$ that is F{\o}lner independent is also locally mixing.

Suppose that $\mu$ satisfies the definition of F{\o}lner independence (\cref{def:FI}) with $\ve=\tilde{\rho}(n)$ for some rate function $\tilde{\rho}$. We shall show that $\mu$ is locally mixing with rate function $2\rho$ given by $\rho(n) := 4\tilde{\rho}(\lfloor \frac{n}{12} \rfloor)$. Thus, we fix $n \ge 1$ and aim to construct a coupling between two samples of $\mu$ with the two properties required by the definition of local mixing (\cref{def:local-mixing}). To avoid measure-theoretic technicalities, we also fix $N \gg n$ and construct the coupling between two samples of $\mu_{|\Lambda_N}$ (with the bound on the probability of disagreement independent of $N$). Taking any subsequential limit of these couplings as $N \to \infty$ will yield the required coupling.
Thus, it suffices to construct a coupling between $f \sim \mu_{|\Lambda_N}$ and $f' \sim \mu_{|\Lambda_n}$ such that $f_{|\Lambda_N \setminus \Lambda_n}$ and $f'$ are independent and $\P(f(v) \neq f'(v)) \le 2\rho(k)$ for any $0 \le k \le n$ and $v \in \Lambda_{n-k}$. In turn, it suffices to construct a measure $\nu$ on $\Lambda_n$ and a coupling between $f \sim \mu_{|\Lambda_N}$ and $f' \sim \nu$ such that $f_{|\Lambda_N \setminus \Lambda_n}$ and $f'$ are independent and $\P(f(v) \neq f'(v)) \le \rho(k)$ for any $0 \le k \le n$ and $v \in \Lambda_{n-k}$.

Throughout the proof, we redefine $\Lambda_k$ to be the box $\{-k+1,\dots,k\}^d$ so that it has side-length $2k$ and volume $(2k)^d$. This is merely for notational convenience, so that $\Lambda_k$ perfectly tiles $\Lambda_{mk}$ for any integer $m \ge 1$. The notions of local mixing and F{\o}lner independence are clearly unaffected by this change. We also let $\Lambda_0$ denote the singleton consisting of the origin.

By the choice of $\tilde{\rho}$, for any $k \ge 0$, there is a collection of couplings $(\pi_k^\tau)_{\tau \in \cA^{\Lambda_{2N} \setminus \Lambda_k}}$ between $\mu(f_{|\Lambda_k} \in \cdot \mid f_{|\Lambda_{2N} \setminus\Lambda_k}=\tau)$ and $\mu|_{\Lambda_k}$ such that $\frac{1}{|\Lambda_k|} \sum_{v \in \Lambda_k} \pi_k^\tau(f(v) \neq f'(v)) \le \tilde{\rho}(k)$ for all $\tau$ but a set of $\mu_{|\Lambda_{2N} \setminus\Lambda_k}$-measure at most $\tilde{\rho}(k)$. By sampling $f_{|\Lambda_{2N} \setminus \Lambda_k}$ from $\mu_{|\Lambda_{2N} \setminus \Lambda_k}$ and then sampling from $\pi_k^{f_{|\Lambda_{2N} \setminus \Lambda_k}}$, this gives a coupling $\pi_k$ of $f \sim \mu_{|\Lambda_{2N}}$ and $f' \sim \mu_{|\Lambda_k}$ such that $f_{|\Lambda_{2N} \setminus \Lambda_k}$ and $f'$ are independent and
\[ \frac{1}{|\Lambda_k|}\sum_{v \in \Lambda_k} \pi_k(f(v) \neq f'(v)) \le 2\tilde{\rho}(k) .\]
We aim to construct such a coupling (with $\Lambda_k$ replaced by $\Lambda_n$) in which a similar such bound holds term by term, not just on average.

We extend the collection $(\pi_k^\tau)$ to include $\tau$ which are defined on any subset of $\Lambda_{2N} \setminus \Lambda_k$, by averaging over the values on the remaining part outside of $\Lambda_k$. That is, if $S \subsetneq \Lambda_{2N} \setminus \Lambda_k$, then for $\tau \in \cA^S$, we define $\pi^\tau_k(\cdot) := \E[\pi_k^{\xi_{|\Lambda_{2N} \setminus \Lambda_k}}(\cdot)]$, where $\xi \sim \mu(\cdot \mid \tau)$.
Observe that for any such $S$, if $\tau \sim \mu_{|S}$, then for any $v \in \Lambda_k$,
\[ \E[\pi^\tau_k(f(v) \neq f'(v))] = \pi_k(f(v) \neq f'(v)) =: p_{k,v} .\]
Note that the above would not necessarily hold if instead of the above averaging we were to appeal to F{\o}lner independence again (which would yield an unrelated $\pi_k^\tau$).

We now also extend the collection $(\pi_k^\tau)$ to allow translates of $\Lambda_k$ as follows.
Let $B=\Lambda_k+b$ be a box centered at $b$ and suppose that $B \subset \Lambda_n$.
For a boundary condition $\tau$ defined on a subset of $S$ of $\Lambda_N \setminus B$, we define $\pi_B^\tau$ to be the coupling between $\mu(f_{|B} \in \cdot \mid f_{|\Lambda_N \setminus B}=\tau)$ and $\mu_{|B}$ obtained by translating $B$ and $\tau$ to the origin, applying the appropriate coupling, and translating back. Precisely, define $\pi_B^\tau(E) := \P((f_{v-b},f'_{v-b})_{v \in B} \in E)$ for any $E \subset \cA^B \times \cA^B$, where $(f,f') \sim \pi_k^{\tau'}$ and $\tau' \in \cA^{S-b}$ is defined by $\tau'_{v-b}:=\tau_v$ for $v \in S$. Note that this is well defined since $\tau'$ is defined on $S-b$ which is a subset of $\Lambda_{2N} \setminus \Lambda_k$, and that this is a coupling between the two claimed measures by the translation-invariance of $\mu$.

Let $\cB=\{B_1,\dots,B_\ell\}$ be a partition of $\Lambda_n$ into boxes (of various sizes). We construct a coupling ${\sf P}_\cB$ of $f \sim \mu$ and $f' \sim \nu_\cB := \mu|_{B_1} \times \cdots \times \mu|_{B_\ell}$ as follows.
Let $k_1,\dots,k_\ell$ be the sizes of the boxes and let $b_1,\dots,b_\ell$ be their centers, so that $B_i=\Lambda_{k_i}+b_i$ for all $i$.
Denote $B_0 := \Lambda_N \setminus \Lambda_n$.
First, sample $f_{|B_0}$.
Next, conditioned on $f_{|B_0}$, sample $(f_{|B_1},f'_{|B_1})$ from $\pi_{B_1}^{f_{|B_0}}$. Now suppose we have already sampled $f$ on $B_0 \cup \cdots \cup B_{i-1}$ and $f'$ on $B_1 \cup \cdots \cup B_{i-1}$, and conditioned on this, sample $(f_{|B_i},f'_{|B_i})$ from $\pi_{B_i}^{f_{|B_0 \cup \cdots \cup B_{i-1}}}$. It is straightforward that this procedure defines a pair $(f,f')$ such that $f \sim \mu$ and $f' \sim \nu_\cB$, and such that $f_{|B_0}$ is independent of $f'$. 
Furthermore,
\begin{equation}\label{eq:folner-indep-coupling}
{\sf P}_\cB(f(v) \neq f'(v)) \le p_{k_i,v-b_i} \qquad\text{for any }1 \le i \le \ell\text{ and }v \in B_i .
\end{equation}

We define a coupling $\sf P$ between $f \sim \mu$ and $f' \sim \nu$ (with $\nu$ defined below) by choosing $\cB$ randomly and then applying ${\sf P}_\cB$ (independently of $\cB$). We construct $\cB$ as follows.
Let $m:=\lfloor \log_4 n \rfloor$ and choose a uniformly random $x \in \Lambda_{4^m}$. For every integer $i$ between $0$ and $m$, and in decreasing order (that is, starting from $i=m$), extend $\Lambda_{4^i}+x$ to a tiling of $\Z^d$ by translates of $\Lambda_{4^i}$, and add to $\cB$ those boxes of the tiling that are at distance at least $4^i$ from $\Lambda_n^c$ and disjoint from all boxes already in $\cB$. At the end of this procedure, any vertex of $\Lambda_n$ that is not finally covered by a box in $\cB$ is added to $\cB$ as a singleton. We also order the boxes in $\cB$ arbitrarily.
This yields a coupling $\sf P$ between $f \sim \mu$ and $f' \sim \nu := \E[\nu_\cB]$.

\begin{figure}
 \centering
 \includegraphics[scale=0.4]{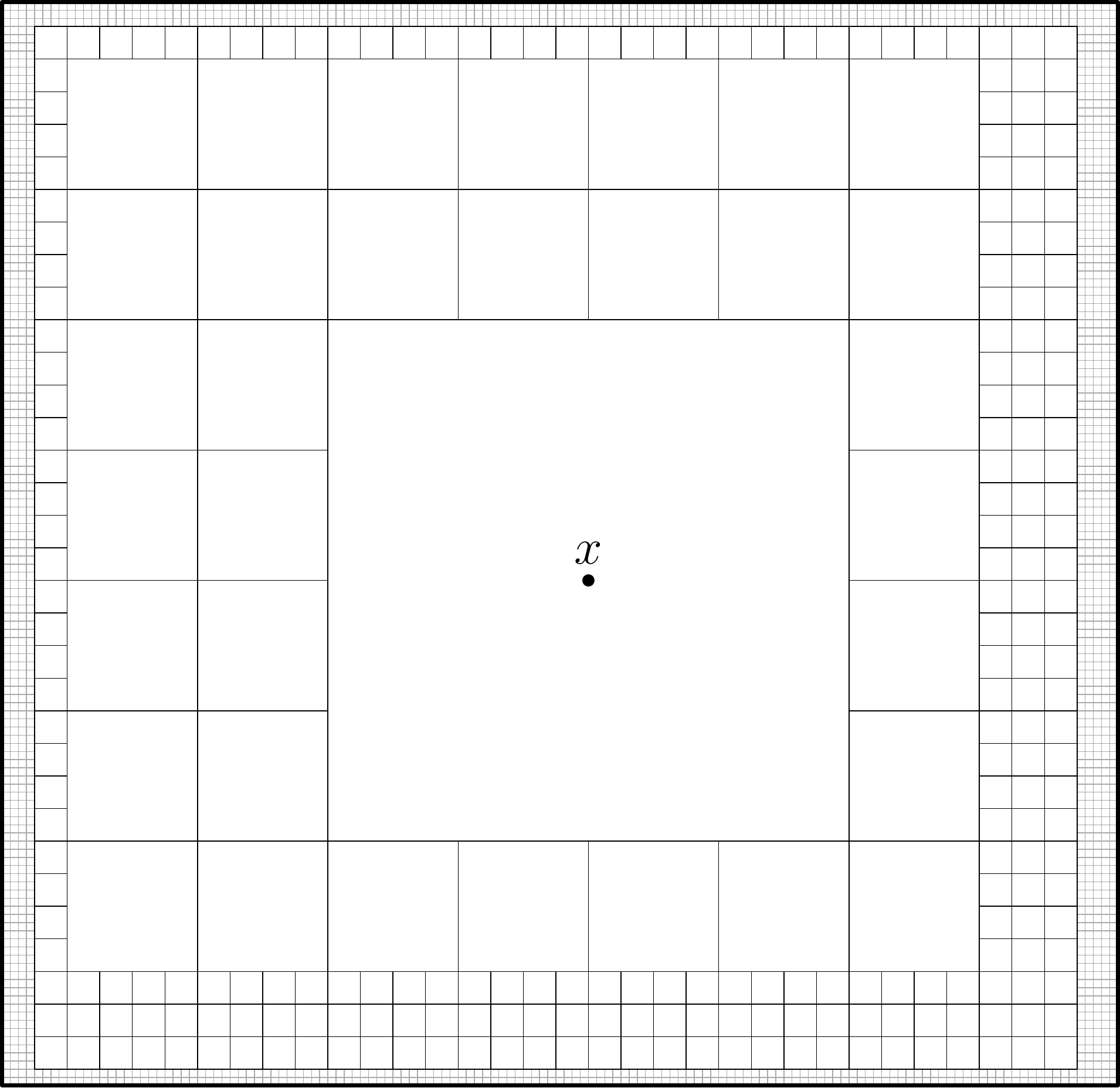}
 \caption{The random partition of $\Lambda_n$ into boxes.}
 \label{fig:box_partition}
\end{figure}

By construction, under this coupling, $f_{|\Lambda_N\setminus\Lambda_n}$ and $f'$ are independent. It remains to show that ${\sf P}(f(v) \neq f'(v)) \le \rho(k)$ for any $0 \le k \le n$ and vertex $v \in \Lambda_{n-k}$.
We may assume that $k \ge 12$ as otherwise $\rho(k)>1$ and there is nothing to prove.

Suppose first that $3 \cdot 4^i < k < 4^{i+1}$ for some $1 \le i \le m$. In this case, whatever $x$ happens to be, $v$ always belong to a box $B_v \in \cB$ of size $4^i$. Since $x$ is chosen uniformly in $\Lambda_{4^m}$ and $4^m$ is a multiple of $4^i$, it is easy to see using~\eqref{eq:folner-indep-coupling} that
\[ {\sf P}(f(v) \neq f'(v)) \le \frac{1}{|\Lambda_{4^i}|} \sum_{u \in \Lambda_{4^i}} p_{4^i,u} \le 2\tilde{\rho}(4^i) \le 2\tilde{\rho}(\tfrac k4) \le \rho(k) .\]
Otherwise, $4^i \le k \le 3 \cdot 4^i$ for some $1 \le i \le m-1$. In this case, depending on the value of $x$, the box $B_v \in \cB$ to which $v$ belongs has size either $4^i$ or $4^{i-1}$. For each $u \in \Lambda_{4^i}$, let $a_u$ denote the number of choices for $x$ such that $B_v$ is a box of size $4^i$ and $v \in u+4^i \Z^d$, and similarly, for each $u \in \Lambda_{4^{i-1}}$, let $b_u$ denote the number of choices for $x$ such that $B_v$ is a box of size $4^{i-1}$ and $v \in u+4^{i-1} \Z^d$. Then, using~\eqref{eq:folner-indep-coupling}, we obtain that
\[ {\sf P}(f(v) \neq f'(v)) \le \frac{1}{|\Lambda_{4^m}|} \left[ \sum_{u \in \Lambda_{4^i}} a_u p_{4^i,u} + \sum_{u \in \Lambda_{4^{i-1}}} b_u p_{4^{i-1},u} \right] .\]
It is not hard to see that there exist $a$ and $b$ such that $a_u \in \{a,a+1\}$ for all $u \in \Lambda_{4^i}$ and $b_u \in \{b,b+1\}$ for all $u \in \Lambda_{4^{i-1}}$. Using the bounds $a_u \le a+1$, $b_u \le b+1$, $a|\Lambda_{4^i}|+b|\Lambda_{4^{i-1}}| \le |\Lambda_{4^m}|$ and $|\Lambda_{4^i}|+|\Lambda_{4^{i-1}}|\le|\Lambda_{4^m}|$, yields that 
\[ {\sf P}(f(v) \neq f'(v)) \le 4\tilde{\rho}(4^{i-1}) \le 4\tilde{\rho}(\tfrac{k}{12}) \le \rho(k) . \qedhere \]

\bigbreak
\paragraph{\textbf{Acknowledgements.}} We thank Nishant Chandgotia, Tom Meyerovitch and Ron Peled for several fruitful discussions. We also thank Tom for suggesting this question. The first author thanks Benoit Laslier for some stimulating conversations.
Research of GR was supported in part by NSERC 50311-57400 and University of Victoria start-up 10000-27458.
Research of YS was supported in part by NSERC of Canada.

\bibliographystyle{amsplain}
\bibliography{library}

\end{document}